\documentclass[12pt]{article}
\usepackage{amsfonts}
\usepackage{amssymb}
\usepackage{mathrsfs}
\usepackage{srcltx}
\usepackage{amsmath}
\usepackage{amsthm}
\usepackage{amstext}
\usepackage{amsopn}
\usepackage{graphicx}
\usepackage{dsfont}
\usepackage{color}
\usepackage{fullpage}
\usepackage{soul}
\usepackage{hyperref}
\usepackage{bm}
\newtheorem{theorem}{Theorem}[section]
\newtheorem{lemma}[theorem]{Lemma}
\newtheorem{proposition}[theorem]{Proposition}
\newtheorem{corollary}[theorem]{Corollary}
\theoremstyle{definition}

\usepackage{caption}
\usepackage{subcaption}

\theoremstyle{remark}
\newtheorem{remark}[theorem]{Remark}

\numberwithin{equation}{section}

\usepackage{bm}

\newcommand{\ba}{\begin{array}}
\newcommand{\ea}{\end{array}}

\newcommand{\Om}{\Omega}

\newcommand{\ds}{\displaystyle}

\begin{document}
\date{}
\title{ \textbf{\large{On  degenerate reaction-diffusion epidemic models with mass action or standard incidence mechanism
%The impact of limiting population movement on the transmission of infectious disease in reaction-diffusion epidemic models
}}
}
\author{
Rachidi B. Salako\\[-1mm]
{\small Department of Mathematical Sciences, University of Nevada  Las Vegas,}\\[-2mm]
{\small Las Vegas, NV 89154}\\
Yixiang Wu\\[-1mm]
%\footnote{Corresponding Author, Email: yixiang.wu@mtsu.edu} \\[-1mm]
{\small Department of Mathematical Sciences, Middle Tennessee State University}\\[-2mm]
{\small Murfreesboro, Tennessee 37132, USA}
}

\maketitle
\begin{abstract}
In this paper, we consider  reaction-diffusion epidemic models with mass action or standard incidence mechanism and study the impact of limiting population movement on disease transmissions. We set either the dispersal rate of the susceptible or infected people  to zero and study  the corresponding degenerate reaction-diffusion model.  Our main approach to study the global dynamics of these models is to construct  delicate Lyapunov functions.   Our results show that the consequences of limiting the movement of susceptible or infected people depend on transmission mechanisms, model parameters, and population size.

\end{abstract}

\section{Introduction}

Various differential equation epidemic models have been proposed to study the spread of infectious diseases \cite{1991book,  2008book, brauer2019mathematical, Diekmann2000, martcheva2015introduction}, and it has been recognized that population mobility  \cite{apostolopoulos2007population, stoddard2009role,tatem2006global} and  the spatial heterogeneity of the environment \cite{hagenaars2004spatial,lloyd1996spatial} are key factors in disease transmissions. In order to address these issues, many reaction-diffusion epidemic models with non-constant coefficients have been proposed  and studied \cite{Allen,Fitz2008,WangZhao,webb1981reaction}. Specific infectious diseases modeled by diffusive models include  malaria
\cite{lou2010periodic,lou2011reaction}, rabies \cite{kallen1985simple,murray1986spatial}, dengue fever \cite{takahashi2005mathematical},
West Nile virus \cite{lewis2006traveling,lin2017spatial}, influenza \cite{magal2020spatial}, Covid-19 \cite{fitzgibbon2021diffusive,kevrekidis2021reaction,setti2020potential,viguerie2020diffusion}, etc.

In this paper, we  consider the following  Susceptible-Infected-Susceptible (SIS) diffusive epidemic model, which is a natural  extension of the classic ordinary differential equation epidemic model  by Kermack and McKendrick \cite{kermack1927contribution}:
\begin{equation}\label{model}
\begin{cases}
\ds\partial_t S=d_S\Delta S-f(x, S, I)+\gamma(x)  I, &x\in\Omega, t>0,\\
\ds\partial_t I=d_I\Delta I+f(x, S, I)-\gamma(x)  I,&x\in\Omega, t>0,\\
\ds\partial_\nu S=\partial_\nu I=0, &x\in\partial\Omega,t>0,\\
S(x, 0)=S_0(x), \ I(x, 0)=I_0(x), &x\in\Omega.
\end{cases}
\end{equation}
Here, the individuals are assumed to live in a bounded domain $\Omega\subset \mathbb{R}^n$ with smooth boundary $\partial\Omega$;
$S(x, t)$ and $I(x, t)$ are the density of susceptible and infected individuals at position $x\in\Omega$ and time $t$, respectively;  $d_S$ and $d_I$ are the movement rates of susceptible and recovered individuals, respectively;  $\gamma$ is the disease  recovery rate; $f(x, S, I)$ describes the interaction of susceptible and infected people; $\nu$ is the unit outward normal of $\partial\Omega$ and the homogeneous Neumann boundary conditions mean that the individuals cannot cross the boundary. 

In the pioneering work by Allen \emph{et al}. \cite{Allen}, model \eqref{model} with standard incidence mechanism, $f(x, S, I)=\beta(x)SI/(S+I)$, has been proposed and studied ($\beta$ is called the disease recovery rate).  In \cite{Allen}, the authors define a basic reproduction number $\mathcal{R}_0$ and show that the model has a unique  endemic equilibrium (EE) (i.e. positive equilibrium) if $\mathcal{R}_0>1$. Most importantly, they show that  the disease component of the EE approaches zero as $d_S$ approaches zero if $\beta-\gamma$ changes sign. Biologically, assuming that the population eventually stabilizes at the EE, this result indicates that  the disease can be eliminated by controlling the mobility of susceptible individuals if there are places that are of low risk (i.e. $\beta(x)<\gamma(x)$). 
In contrast if the dispersal rate of infected people is limited, \cite{Peng2009}  shows that the disease cannot be completely eliminated. In 
\cite{DengWu, WuZou}, the authors  considered 
 model \eqref{model} with mass action mechanism, $f(x, S, I)=\beta(x)SI$, which is algebraically simpler but mathematically more challenging.  For this model, assuming  again that the population eventually stabilizes at  EE solutions, it has been shown   that  lowering the movement of susceptible people can eliminate the disease only when the size of the population is below some critical number, solely determined by $\gamma/\beta$,   \cite{castellano2022effect,WuZou}, while infected individuals may concentrate on certain hot spots when  limiting their movement  \cite{castellano2022effect,peng2023novel,WuZou}. It is important to note that stability of the EE solutions for models with standard incidence or mass action mechanism is only known in a few cases : either the population movement rate is uniform (i.e., $d_S=d_I$), or the ratio $\beta/\gamma$ is constant \cite{DengWu, PengLiu}. For more related works, we refer the interested readers to 
\cite{ChenShi2020,CuiLamLou,CuiLou,Jiang,KuoPeng,LiPeng,Li2018, li2020dynamics,LouSalako2021,LSS2023,PengLiu,peng2021global,Peng2013, PengZhao,tao2023analysis,
Tuncer2012,wen2018asymptotic} and the references therein.

For the corresponding ordinary differential equation epidemic model of \eqref{model}:
\begin{equation*}
\begin{cases}
S'=-f(S, I)+\gamma  I,\\
I'=f(S, I)-\gamma  I,
\end{cases}
\end{equation*}
the global dynamics is determined by the basic reproduction number, which can be interpreted as the average number of infected individuals generated by one infectious individual {in an otherwise susceptible population}: if it is greater than one, the solution converges to an EE and the disease persists; if it is less than one, the solution converges to a disease free equilibrium and the infected individuals go to extinction. Here, the basic reproduction number is $\mathcal{R}^1_0=N\beta/\gamma$ if $f(S, I)=\beta SI$ and $\mathcal{R}_0^2=\beta/\gamma$ if $f(S, I)=\beta SI/(S+I)$.

In this paper, we  revisit model \eqref{model} and study the impact of limiting population movement   on disease transmissions. Different from most of the aforementioned studies, we will work on the global dynamics of the time dependent model \eqref{model} with either  $d_S=0$ or $d_I=0$  rather than consider the asymptotic profiles of the EE as $d_S\to 0$ or $d_I\to 0$.  { Intuitively, if the disease 
%evolution happens in
evolves at  a faster time  scale than the control of population movement, and the 
%solutions of  model \eqref{model}  establish
population eventually stabilizes at an EE, then the asymptotic profiles of the EE solutions % will 
may reflect the effect of the control strategy.} However if the control of population movement happens in a faster time scale and the solution of  model  \eqref{model} converges to that of the corresponding degenerate system as $d_S\to 0$ or $d_I\to 0$, then the global dynamics of the degenerate system will better tell the impact of the control strategies. 

% {   In this work, we assume that control strategies happen on a fast scale and seek to understand their impacts on  the global dynamics of the disease. Hence, our study complements the existing work and offers an alternative approach to studying the effect of restricting population movement on the dynamics of infectious diseases. Note that when either $d_S=0$ or $d_I=0$, system \eqref{model} becomes degenerates and hence there is lack of some regularities of the solution operator. This induces several difficult challenging in the study of the large time behavior of solutions. }  Our approach to study the global dynamics of the degenerate system \eqref{model} (i.e $d_S=0$ or $d_I=0$) is to construct Lyapunov functions and combine these with delicate analysis.  

% {  There are also several studies on some other types of degenerate reaction-diffusion population models. We refer the interested readers to (\cite{DomateSalako2022, EfendievOtaniEberl2022, LouSalako2020, LouTaoWinkler2014, LouWinkler2020,   OnyidoSalako2023} and the references cited therein).}

There are  several recent efforts on  degenerate reaction-diffusion population models (see \cite{DomateSalako2022, EfendievOtaniEberl2022, LouSalako2020, LouTaoWinkler2014, LouWinkler2020,   OnyidoSalako2023,wu2018dynamics} and the references cited therein). In particular, the two works \cite{castellano2022effect, LouSalako2020} have partial results on model \eqref{model} with $d_S=0$, which was interpreted as the situation of a total lock down for the susceptible population. Note that when either $d_S=0$ or $d_I=0$, system \eqref{model} becomes degenerate and hence there is a lack of some compactness of the solution operator, which induces many challenges in the study of the large time behavior of the solutions.   

Our main approach to study the global dynamics of the degenerate system \eqref{model} (i.e $d_S=0$ or $d_I=0$) is to construct Lyapunov functions (except for the case $d_S=0$ in subsection \ref{subsection31}) and combine these with delicate analysis.  In particular, in subsection \ref{subsection32}, we present a Lyapunov function that one may easily draw a false conclusion using it and the convergence result based on it is very quite unusual (see Remark \ref{remark_ly} and Fig. \ref{fig2}). In the proof of Theorem \ref{theorem_stddi}, we construct a Lyapunov function $V$ that does not satisfy $\dot V\le 0$ (see Eq. \ref{vvv}) but we are still able to conclude the convergence of the solution.

The rest of our paper is organized as follows. In section 2, we list the assumptions and terminology and present some useful results. In section 3, we consider the model with mass action mechanism with $d_S=0$ or $d_I=0$. In section 4, we consider the model with standard incidence mechanism. In section 5, we run some numerical simulations to illustrate the results. In section 6, we compare the results from the two different mechanisms and control strategies and discuss the implications for disease control.

\section{Preliminaries}

Throughout the paper, we make the following assumptions on the parameters: 
\begin{enumerate}
\item[(A1)] The functions $\beta, \gamma$ are positive and H\"older continuous on $\bar\Omega$;
\item[(A2)] The functions $S_0, I_0$ are nonnegative and continuous on $\bar\Omega$ with $I_0\not\equiv 0$. Moreover,  $\int_\Omega (S_0+I_0)dx=N$ for some fixed positive constant $N$.
\end{enumerate}
Integrating the first two equations of model \eqref{model} over $\Omega$ and summing up them, we find that 
$$
\frac{d}{dt} \int_\Omega (S(x, t)+I(x, t))dx=0,
$$
which means the total population $\int_\Omega (S(x, t)+I(x, t))dx$ remains a constant for all $t\ge 0$. By assumption (A2), the total population is $N$.

 For convenience, we introduce a few definitions and notations. Set $r:=\gamma/\beta$ and $R:=\beta/\gamma$.  For a  real-valued { continuous} function $h$ on $\overline\Omega$,  let $h_M:=\sup_{x\in\Omega} h(x)$, $h_m:=\inf_{x\in\Omega} h(x)$, and $\bar h=\int_\Omega h(x)dx/|\Omega|$.

Let  $\{e^{t\Delta}\}_{t\ge 0}$ be the analytic $c_0$-semigroup on $L^p(\Omega)$,  $1\le p<\infty$, generated by the Laplace operator $\Delta$ on $\Omega$ subject to the homogeneous Neumann boundary conditions on $\partial\Omega$.  Let ${\rm Dom}_{p}(\Delta)$ be the domain of  the infinitesimal generator $\Delta$ of $\{e^{t\Delta}\}_{t\ge 0}$ on $L^p(\Omega)$. Then ${\rm Dom}_{p}(\Delta)=\{u\in W^{2,p}(\Omega) | \partial_{\nu}u=0\ \text{on}\ \partial\Omega\}$ for $p\in(1,\infty)$ and $ {\rm Dom}_{1}(\Delta)\subset \{u\in W^{2,1}(\Omega) | \partial_{\nu}u=0\ \text{on}\ \partial\Omega\}$ (see \cite{Amann1983}).    When $\{e^{t\Delta}\}_{t\ge 0}$ is considered as an analytic $c_0$-semigroup on $C(\bar{\Omega})$, the domain of its infinitesimal  generator $\Delta$ is given by  
$$
{\rm Dom}_{\infty}(\Delta):=\left\{u\in\cap_{p\ge 1}{\rm Dom}_{p}(\Delta)| \ \Delta u\in C(\overline\Omega)\right\}.
$$

% Given a real number $p\ge 1$ and a positive integer $k$, let $L^p(\Omega)$ and $W^{k,p}(\Omega)$ denote the  Banach spaces of the $L^p$-integrable functions and the Sobolev space, respectively, and  
Let $\mathcal{Z}:=\{u\in L^1(\Omega) : \int_{\Omega}udx=0\}$ be a Banach subspace of $L^1(\Omega)$. Then $\{ e^{t\Delta}\}_{t\ge 1}$ leaves $\mathcal{Z}$ invariant and by \cite[Lemma 1.3]{winkler2010aggregation}, there is a positive real number $C_0>0$ such that 
\begin{equation}\label{KBB0}
    \|e^{t\Delta}u\|_{L^{1}(\Omega)}\le C_0e^{-\sigma_1 t}\|u\|_{L^1(\Omega)}, \quad \forall\ u\in \mathcal{Z},
\end{equation}
where $\sigma_1$ is the first positive eigenvalue of $-\Delta$, subject to the homogeneous boundary condition on $\partial\Omega$.  By \eqref{KBB0}, the restriction $\Delta_{|\mathcal{Z}}$ of $\Delta$ on $\mathcal{Z}\cap{\rm Dom}_1(\Delta)$ is invertible. Hence, for any $0\le \alpha<1$, the fractional power space { $\mathcal{Z}^{\alpha}$} of $-\Delta_{|\mathcal{Z}}$ is well defined (see \cite{DanHenry}).  Denote by $\{e^{t\Delta_{|\mathcal{Z}}}\}_{t\ge 0}$, the restriction of the $\{e^{t\Delta}\}_{t\ge 0}$ on $\mathcal{Z}$, then it is also an analytic $c_0$-semigroup. The following estimates on $\mathcal{Z}^{\alpha}$ will be needed:

\begin{lemma}\label{fractional-power-space-estimates}\cite[Theorem 1.4.3]{DanHenry} For any $0< \alpha<1$, there is $C_{\alpha}>0$ such that 
    \begin{equation}\label{KBB1}
        \|e^{t\Delta_{|\mathcal{Z}} }z\|_{\mathcal{Z}^{\alpha}}\le C_{\alpha}t^{-\alpha}e^{-\sigma_1 t}\|z\|_{L^1(\Omega)},\quad \forall\ t>0,\  z\in\mathcal{Z},
    \end{equation}
    and 
    \begin{equation}\label{KBB2}
        \|(e^{t\Delta_{|\mathcal{Z}}}-{\rm id})z\|_{L^1(\Omega)}\le C_{\alpha}t^{\alpha}\|z\|_{\mathcal{Z}^{\alpha}},\quad \forall\ z\in \mathcal{Z}^{\alpha}.
    \end{equation}   
\end{lemma}

The following well-known result will be used after the construction of Lyapunov functions and we provide a proof for convenience. 
\begin{lemma}\label{lemma_inf0}
    Suppose that $\phi:\mathbb{R}_+\to\mathbb{R}_+$ is  H\"older continuous and $\int_0^\infty\phi(t)dt<\infty$. Then $\lim_{t\to\infty}\phi(t)=0$.
\end{lemma}
\begin{proof}
 {    Suppose to the contrary that $\lim_{t\to\infty}\phi(t)\neq 0$. Then there exists $\epsilon>0$ and an increasing sequence of nonnegative numbers $\{t_k\}$ converging to infinity such that $\phi(t_k)>\epsilon$ for all $k\ge 1$. Restricting to a subsequence if necessary, we assume that $t_{k+1}-t_k>1$ for each $k\ge 1$. Since $\phi:\mathbb{R}_+\to\mathbb{R}_+$ is  H\"older continuous, there exist $\alpha\in (0, 1)$ and $M>0$ such that $|\phi(t)-\phi(s)|\le M|t-s|^\alpha$ for all $t, s\ge 0$. Let $\delta=\min\{(\frac{\epsilon}{2M})^\frac{1}{\alpha}, 1\}$. Then $\phi(t)>\epsilon/2$ for all $t\in [t_k, t_k+\delta]$ and $k\ge 1$. It follows that 
    $$
    \int_0^\infty \phi(t)dt\ge \sum_{k\ge 1}\int_{t_k}^{t_k+\delta}\phi(t)dt\ge \sum_{k\ge1} \delta \frac{\epsilon}{2}=\infty,
    $$
    which is a contradiction. }
\end{proof}

We will  need the following Hanack type inequality (see \cite{huska2006harnack}):
\begin{lemma}\label{lemma_harnack}
Let $u$ be a  nonnegative solution of the following problem on $\Omega\times (0, T)$:
\begin{equation*}\label{model-aux}
\begin{cases}
\displaystyle\partial_t u=\Delta u+ a(x, t) u,&x\in\Omega, t>0,\\
\partial_\nu u=0, &x\in\partial\Omega,t>0,
\end{cases}
\end{equation*}
where $a\in L^\infty(\Omega\times (0, \infty))$. Then for any $0<\delta<T$, there exists $C>0$ depending on $\delta$ and $\|a\|_{L^\infty(\Omega\times (0, \infty))}$ such that 
$$
\sup_{x\in\Omega} u(x, t)\le C\inf_{x\in\Omega} u(x, t), \ \ \text{for all } t\in [\delta, T).
$$
\end{lemma}

We also recall the following well known result about the elliptic eigenvalue problem.

\begin{lemma}\label{eig-value-lemma}\cite{CC2003}
 Suppose that $d>0$  and  $h\in C(\overline{\Omega})$. Let $\sigma(d,h)$ be the principal eigenvalue of the following  elliptic eigenvalue problem:
\begin{equation*}
    \begin{cases}
    d\Delta \varphi +h\varphi=\sigma \varphi, & x\in\Omega,\cr
    \partial_\nu \varphi=0, & x\in\partial\Omega.
    \end{cases}
\end{equation*}
Then $\sigma(d,h)$ is  simple, associated with a positive eigenfunction, and given by the variational formula
\begin{equation}\label{princ-formula}
    \sigma(d,r)=\sup\Big\{\int_{\Omega}\big[h\varphi^2-d|\nabla\varphi|^2\big]dx:\ \varphi\in W^{1,2}(\Omega)\ \text{and}\  \|\varphi\|_{L^2(\Omega)}=1\Big\}.
\end{equation}
Furthermore, the following conclusions hold:
\begin{itemize}
\item[(i)]If $h(x)\equiv h$ is a constant function, then $\sigma(d,h)=h$ for all $d>0$.
\item[(ii)]If $h(x)$ is not constant, then the map $(0,\infty)\ni d\mapsto \sigma(d,h)$ is strictly decreasing with
\begin{equation}\label{ga9}
    \lim_{d\to 0^+}\sigma(d,h)=h_{M} \quad \text{and}\quad {\lim_{d\to\infty}\sigma(d,h)=\overline{h}}.
\end{equation}
%\item[(iii)] If $h_1\le,\ne h_2$, then $\sigma(d,h_1)<\sigma(d,h_2)$ for every $d>0$.
\end{itemize}
\end{lemma}

\section{Model with mass action mechanism}
% In this section, we will consider the impact of limiting the movement of susceptible and infected people in a reaction-diffusion with mass action mechanism:
% \begin{equation}\label{model-mass}
% \begin{cases}
% \ds\f{\partial  S}{\partial t}=d_S\Delta S-\beta(x) S I+\gamma(x)  I, &x\in\Omega, t>0,\smallskip\\
% \ds\f{\partial   I}{\partial t}=d_I\Delta I+\beta(x) S I-\gamma(x)  I,&x\in\Omega, t>0,\smallskip\\
% \ds\f{\partial S}{\partial \nu}=\ds\f{\partial I}{\partial \nu}=0, &x\in\partial\Omega,t>0.
% \end{cases}
% \end{equation}

\subsection{Limiting the movement of susceptible people}\label{subsection31}
First, we consider the impact of limiting the movement of susceptible people on \eqref{model} with  mass action mechanism, that is, the long time behavior of the following degenerate system:
\begin{equation}\label{model-mass0}
\begin{cases}
\partial_t S=-\beta(x) S I+\gamma(x)  I, &x\in\bar\Omega, t>0,\\
\partial_t I=d_I\Delta I+\beta(x) S I-\gamma(x)  I,&x\in\Omega, t>0,\\
\partial_\nu I=0, &x\in\partial\Omega,t>0,\\
S(x, 0)=S_0(x), \ I(x, 0)=I_0(x), &x\in\bar\Omega.
\end{cases}
\end{equation}

The following result states that the solution of \eqref{model-mass0} exists and is bounded:
\begin{proposition}\label{prop}
Suppose that {\rm (A1)-(A2)} holds and $d_I>0$. Then \eqref{model-mass0} has a unique nonnegative global solution $(S, I)$, where
$$S\in C^1([0, \infty), C(\overline{\Omega}))\quad  \text{and}\quad I\in C([0, \infty), C(\overline{\Omega}))\cap C^1((0, \infty), {\rm Dom}_{\infty}(\Delta)).$$  Moreover, there exists $M>0$ depending on (the $L^\infty$ norm of) initial data such that 
\begin{equation}\label{bound0}
\|S(\cdot, t)\|_{L^\infty(\Omega)}, \ \|I(\cdot, t)\|_{L^\infty(\Omega)}\le M, \ \forall t\ge 0.
\end{equation}
\end{proposition}
\begin{proof}
First we suppose that the solution exists  on some maximal interval $[0, t_{max})$, $t_{\max}\in(0,\infty]$ and prove the boundedness of it. Let $M_1=\max\{\|S_0\|_{L^\infty(\Omega)}, \ r_M\}$. By the first equation of \eqref{model-mass0}, $S(x, t)\le M_1$ for all $x\in\bar\Omega$ and $0\le t<t_{\max}$. Since $\int_\Omega Idx\le \int_\Omega(S+I)dx\le N$ for all $t\in[0,t_{\max})$, by \cite[Theorem 3.1]{alikakos1979application} (or  Lemma \ref{lemma_harnack}), there exists $M_2>0$ depending on $\|I_0\|_{L^\infty(\Omega)}$ and $M_1$ such that $\|I(\cdot, t)\|_{L^\infty(\Omega)}\le M_2$ for all $0\le t< t_{\max}$. This proves \eqref{bound0}. 

The integral form of \eqref{model-mass0} is 
\begin{equation}\label{model-int}
\begin{cases}
S(\cdot, t)=e^{-at}S_0+\int_0^t e^{-a(t-s)}(aS(\cdot,s)-\beta S(\cdot, s)I(\cdot, s)+\gamma I(\cdot, s))ds, &t>0,\\
I(\cdot, t)=e^{t(d_I\Delta-a)}I_0+\int_0^t e^{(t-s)(d_I\Delta-a)}(a+\beta S(\cdot, s)-\gamma)I(\cdot, s)ds,&t>0,
\end{cases}
\end{equation}
where $a>0$ is chosen to be sufficiently large. Using Banach fixed point theory, one can show that \eqref{model-int} has a unique nonnegative solution $(S, I)\in [C([0, T], C(\overline{\Omega}))]^2$ for some $T>0$. By the first equation of \eqref{model-int},  $S\in C^1([0, T], C(\overline{\Omega}))$. By \cite[Theorem 4.3.1 and Corollary 4.3.3]{pazy2012semigroups}, $I\in C([0, T], C(\overline{\Omega}))\cap C^1((0, T], {\rm Dom}_{\infty}(\Delta))$. The a priori bound \eqref{bound0} enables us to extend the solution globally.
\end{proof}

We note that the case when the total population is small (i.e., $N< \int_\Omega r dx$) has been addressed in the literature:
\begin{theorem}\cite[Theorem 2.8-(i)]{castellano2022effect}
Suppose that {\rm (A1)-(A2)} holds and $d_I>0$. Let $(S, I)$ be the solution of \eqref{model-mass0}. If $N< \int_\Omega r dx$, then $(S(x, t), I(x, t))\to (S^*(x), 0)$   uniformly in $x\in\bar\Omega$ as $t\to \infty$, where $S^*\in C(\bar\Omega)$ and $\int_\Omega S^*dx=N$.
\end{theorem}

We are ready to study the asymptotic behavior of the solution of \eqref{model-mass0}. 

% Define 
% $$
% V_1(S, I)=\int_\Omega\left(\frac{\beta}{2\gamma} S^2+I \right)dx,\ \ \text{and}\ \ V_2(S, I)=\frac{1}{2}\int_\Omega I^2dx
% $$
% If $(S, I)$ is the solution of \eqref{model-mass0}, it is easy to check that 
% \begin{equation}\label{ly00}
% \frac{d}{dt}V(S, I)=-\int_\Omega \frac{\beta^2}{\gamma}\left(S-r\right)^2Idx.
% \end{equation}

\begin{theorem}\label{theorem_massds}
Suppose that {\rm (A1)-(A2)} holds and $d_I>0$. Let $(S, I)$ be the solution of \eqref{model-mass0}.  Then $\|S(\cdot, t)-S^*\|_{L^\infty(\Omega)}\to 0 $ as $t\to\infty$ for some $S^*\in C(\overline{\Omega})$, and exactly one of the following two statements hold:
\begin{enumerate}
\item[\rm (i)] $S^*=\lambda^*S_0+(1-\lambda^*)r$ for some $\lambda_*\in C(\bar{\Omega})$ with $0< \lambda_*<1$ and  $\sigma(d_I,\beta\lambda^*(S_0-r))\le 0$,  $\int_{\Omega}S^*dx=N$,  and  $\|I(\cdot, t)\|_{L^\infty(\Omega)}\to 0$ as $t\to\infty$;

\item[\rm (ii)]  $S^*=r$  and  $\|I(\cdot,t)-I^*\|_{L^\infty(\Omega)}
\to 0$ as $t\to\infty$, where $I^*={(N-\int_{\Omega}r dx)}/{|\Omega|}$ is a positive constant.
\end{enumerate}
% Then the following statements hold:
% \begin{enumerate}
% \item[{\rm (1)}] If $N\le \int_\Omega r dx$, then $(S(x, t), I(x, t))\to (S^*(x), 0)$   uniformly in $x\in\bar\Omega$ as $t\to \infty$, where $\int_\Omega S^*dx=N$;
% \item[{\rm (2)}] If $N> \int_\Omega r dx$, then $S(x, t)\to r(x)$  and $I(x, t)\to (N-\int_\Omega r dx)/|\Omega|$ uniformly in $x\in\bar\Omega$ as $t\to\infty$ provided that one of the following holds: (i) $S_0\le  r $; (ii)  $S_0\ge  r $; (iii) $\beta$ is constant. {  Rachidi: Can we get rid of these  restrictions?}
% \end{enumerate}
% Moreover, {\rm (i)} holds if $N\le \int_\Omega r dx$; {\rm (ii)} holds if either (1) $N> \int_\Omega r dx$ and  $\beta$ is constant,  or (2) $N>\int_{\Omega}\max\{S_0,r\}$.
Moreover, {\rm (i)} holds if $N\le \int_\Omega r dx$ while  {\rm (ii)} holds if  $N> N^*_{S_0,r}$, where $N^*_{S_0,r}$ is defined by 
\begin{equation}\label{N-critic}
     N^*_{ S_0,r}:=\sup\Big\{\int_{\Omega}(\lambda^*S_0+(1-\lambda^*)r)dx\ :\ \lambda^*\in C(\bar{\Omega};[0,1])\,\,   \text{and}\,\, \sigma(d_I,\beta\lambda^*(S_0-r))\le 0\Big\}.
 \end{equation}
 
% where 
% \begin{equation}\label{N-critic}
%     N^*_{r,S_0}:=\sup\Big\{\int_{\Omega}(\lambda^*S_0(x)+(1-\lambda^*)r(x))dx\ :\ \lambda^*\in C(\bar{\Omega}) \ \text{and}\ 0\le \lambda^*\le 1\Big\}.
% \end{equation}
% one of the following conditions holds: (1) $S_0\le  r $; (2)  $S_0\ge  r $; (3) $\beta$ is constant.
\end{theorem}
% {  Wu: is $N_{r, S_0}^*\equiv\int_\Omega \max\{S_0, r\}$? Just need to define $\lambda^*(x)=\chi_{S_0-r}$, which is not continuous. Then take $\lambda^*_n$ continuous, between 0 and 1, such that $\lambda_n^*\to \lambda^*$.

% Salako: Actually, I thought about this question before. I though that I could easily construct an example. I just realized from my construction that we just need have $\lambda^*_n\to \lambda=\chi_{S_0-r}$. So, you are right. I would suggest that we just modify the result, accordingly. If, necessary, just make a comment to indicate it.

% } 

\begin{proof}
% Fix $x\in\bar\Omega$. If $S_0(x)<r(x)$, then $\partial_tS(x, t)=-(\beta S-\gamma)I>0$ and $S(x, t)$ is strictly increasing for all $t>0$; If $S_0(x)>r(x)$, then $\partial_tS(x, t)=-(\beta S-\gamma)I<0$ and $S(x, t)$ is strictly decreasing for all $t>0$. Hence, $S(x, t)\to S^*(x)$ for each $x\in\bar\Omega$ as $t\to\infty$ for some $S^*\in L^\infty(\Omega)$. By the Lebesgue Dominated Convergence theorem, $S(\cdot, t)\to S^*$ in $L^1(\Omega)$ as $t\to\infty$. {  We claim that $S^*\in C(\overline{\Omega})$. To this end, set  
Define 
$$
J(x,t)=\int_0^{t}I(x,\tau)d\tau,\quad \forall\ (x, t)\in\bar{\Omega}\times [0, \infty).
$$
It is clear that $J(x,t)$ is strictly monotone increasing in $t$ for each $x\in\bar{\Omega}$. Hence we can define
$$
\hat J(x):=\int_{0}^{\infty}I(x,\tau)d\tau=\lim_{t\to\infty}J(x,t) \in(0,\infty],\quad \forall\ x\in\bar{\Omega}.
$$
Now, we distinguish two cases.
\medskip

\noindent {\bf Case 1.} $\hat J(x_0)<\infty$ for some $x_0\in \overline{\Omega}$. By \eqref{bound0},  we have $\sup_{t\ge 0}\|\beta S(\cdot,t)-\gamma\|_{\infty}<\infty$. So by Lemma \ref{lemma_harnack}, there is a positive constant $c_1$ such that 
\begin{equation}\label{TSS2}
I_M(\cdot,t)\le c_1I_{m}(\cdot,t),\quad \forall\ t\ge 1.
\end{equation}
It follows that 
$$
\hat J(x)=J(x,1)+\int_{1}^{\infty}I(x,\tau)d\tau\le J(x,1)+c_1\int_{1}^{\infty}I(x_0,\tau)d\tau\le J(x,1)+c_1J(x_0),\quad \forall\ x\in\bar{\Omega}. 
$$
This shows that $\hat J(x)<\infty$ for all $x\in\bar{\Omega}$ and $\|\hat J\|_{L^{\infty}(\Omega)}<\infty$. Moreover, by \eqref{TSS2}, 
$$ 
\|\hat J-J(\cdot,t)\|_{L^{\infty}(\Omega)}\le c_1\int_{t}^{\infty}I(x_0,\tau)d\tau\to 0 \quad \text{as}\quad t\to\infty.
$$
Hence, $\hat J\in C(\bar{\Omega})$.  Next, observing that 
 $$
 \partial_t(S-r)=\partial_tS=-\beta(S-r)I, \qquad \forall \ (x, t)\in\bar\Omega\times (0, \infty),
 $$
 we have
 \begin{equation}\label{TSS1}
 S(x,t)-r(x)=(S_0(x)-r(x))e^{-\beta(x)J(x,t)},\quad \forall\  (x, t)\in\bar\Omega\times (0, \infty).
 \end{equation}
 This implies 
 $$
 S(\cdot, t)\to S^*:=r+e^{-\beta \hat J}(S_0-r)  \ \text{uniformly on}\ \bar\Omega \ \text{as}\ t\to\infty.
 $$
 Let $\lambda^*=e^{-\beta \hat J}$. Then $0<\lambda^*<1$, $\lambda^*\in C(\bar{\Omega})$, and $S^*=\lambda^*S_0+(1-\lambda^*)r$.  Let $\varphi^*>0$ be the eigenfunction associated with $\sigma(d_I,\beta(S^*-r))$ satisfying $\varphi^*_{M}=1$. Then it holds that 
 \begin{align*}
     \frac{d}{dt}\int_{\Omega}\varphi^*Idx&=d_I\int_{\Omega}\varphi^*\Delta Idx +\int_{\Omega}\beta(S-r)\varphi^*Idx\cr 
=&\sigma(d_I,\beta(S^*-r))\int_{\Omega}\varphi^*Idx+\int_{\Omega}\beta(S-S^*)\varphi^*Idx\cr 
     \ge& \Big(\sigma(d_I,\beta(S^*-r))-\beta_M\|S-S^*\|_{L^{\infty}(\Omega)}\Big)\int_{\Omega}\varphi^*Idx.
 \end{align*}
 Hence, we have
 $$
 N\ge \int_{\Omega}\varphi^*I(x,t)dx\ge e^{\int_{0}^t(\sigma(d_I,\beta(S^*-r))-\beta_M\|S-S^*\|_{L^{\infty}(\Omega)})d\tau}\int_{\Omega}\varphi^*I_0dx,\quad \forall\ t>0.
 $$
 This implies that 
 $$
 \frac{\beta_M}{t}\int_{0}^{t}\|S(\cdot,\tau)-S^*\|_{L^{\infty}(\Omega)}d\tau+\frac{\ln(N)-\ln(\|\varphi^*I_0\|_{L^1(\Omega)})}{t}\ge \sigma(d_I,\beta(S^*-r)),\quad t>0.
 $$
 Observing that the left-hand side of this inequality tends zeo as $t\to\infty$, we conclude that $\sigma(d_I,\beta(S^*-r))\le 0$.

 Finally, by \eqref{bound0} and the parabolic estimates and the Sobolev emdedding theorem,  $I$ is H\"older continuous on $\bar\Omega\times [1, \infty)$. Hence, by Lemma \ref{lemma_inf0} and the fact that $\hat J(x_0)<\infty$, we obtain  from \eqref{TSS2} that
 $$
 \|I(\cdot,t)\|_{L^\infty(\Omega)}\le c_1I(x_0,t)\to 0\quad \text{as} \quad t\to\infty,
 $$
 which in turn implies that $$\int_{\Omega}S^*dx=\lim_{t\to\infty}\int_{\Omega}Sdx=\lim_{t\to\infty}\int_{\Omega}(S+I)dx=N.$$ 
 This completes the proof of {\rm (i)}.

\noindent {\bf Case 2.} $\hat J(x)=\infty$ for all $x\in\bar{\Omega}$. Fix $x_1\in\bar\Omega$. Then $\int_1^{t}I(x_1,\tau)d\tau\to \infty$ as $t\to\infty$. This together with \eqref{TSS2}-\eqref{TSS1} implies that
$$
\|S(\cdot,t)-r\|_{L^{\infty}(\Omega)}\le \|S_0-r\|_{L^{\infty}(\Omega)}e^{-\frac{\beta_m}{c_1}\int_1^{t}I(x_1,\tau)d\tau}\to 0\quad \text{as}\quad t\to\infty.
$$
As a result,
\begin{equation}\label{TSS4}
    \lim_{t\to\infty}\int_{\Omega}Idx=N-\int_{\Omega}rdx.
\end{equation}
This shows that, in the current case,  we must have that $N\ge \int_{\Omega}rdx$. By \eqref{bound0}, the parabolic estimates and the Sobolev embedding theorem, the orbit $\{I(\cdot,t)\}_{t\ge 1}$ is precompact in $C(\bar\Omega)$. Noticing $S(\cdot, t)\to r$ in $C(\bar\Omega)$ as $t\to\infty$, the $\omega$ limit set  $\omega(S_0, I_0)=\cap_{t\ge 1} \overline{\cup_{s\ge t}\{(S(\cdot, t), I(\cdot, s))\}}$ is well-defined, where  the completion is in $[C(\bar\Omega)]^2$. {  Fix   $(S^*,I^*)\in \omega(S_0,I_0)$. Since $S(\cdot,t)\to r$ in $C(\bar{\Omega})$ as $t\to\infty$, then  $S^*=r^*$. Next, we show that $I^*(x)=(N-\int_{\Omega}r)/|\Omega|$ for all $x\in\Omega$. To this end, since $\omega(I_0,S_0)$ is invariant under the semiflow of the solution operator induced by \eqref{model-mass0} and the orbit $\{I(\cdot,t)\}_{t\ge 1}$ is precompact in $C(\bar{\Omega})$, we can employ standard parabolic regularity arguments to the equation of $I(\cdot,t)$, and coupled with the fact that $S(\cdot,t)\to r$ in $C(\bar{\Omega})$ as $t\to\infty$, to conclude that  there is a bounded entire solution $(\tilde{S}(x,t),\tilde{I}(x,t))$ of \eqref{model-mass0} fulfilling $\tilde{I}(\cdot,0)=I^*$ and $\tilde{S}(\cdot,t)=r$ for all $t\in\mathbb{R}$. %{By the invariant property of $\omega(S_0, I_0)$, for any $(S^*, I^*)\in \omega(S_0, I_0)$, there exists a complete orbit of  \eqref{model-mass0} through $(S^*, I^*)$. It is clear that the $S-$component of the complete orbit is  $r$.}
Hence, $\tilde{I}(x,t)$ satisfies 
$$
\begin{cases}   \partial_t\tilde{I}=d_I\Delta\tilde{I}, & x\in\Omega, \ t\in\mathbb{R},\cr 
    \partial_{\nu}\tilde{I}=0, & x\in\partial\Omega,\ t\in\mathbb{R},\cr 
    \tilde{I}(x,0)=I^*(x), & x\in\bar{\Omega}.
\end{cases}
$$%there exists a bounded entire solution of $\partial_t \tilde I=d_I\Delta \tilde I$ with homogeneous Neumann boundary condition such that $\tilde I(\cdot, 0)=I^*$. 
So, $\tilde{I}(\cdot,t)=\int_{\Omega}I^*(x)dx/|\Omega|$ for all $t\in\mathbb{R}$. However by \eqref{TSS4}, $\int_{\Omega}I^*dx=(N-\int_{\Omega}rdx)/|\Omega|$. Hence,  $I^*$ is the  constant function $(N-\int_\Omega r dx)/|\Omega|$.  This shows that $\omega(S_0,I_0)=\{r,(N-\int_{\Omega}rdx)/|\Omega|\}$    
 % \begin{equation}\label{iff}
 %    \lim_{t\to\infty}   \Big\|I(\cdot,t)-\frac{N-\int_{\Omega}rdx}{|\Omega|}\Big\|_{L^\infty(\Omega)}=0,
 % \end{equation}
 and completes the proof of {\rm (ii)}.}

It is easy to see that if $N\le \int_\Omega rdx$, then (i) holds.  
  Note that if {\rm (i)} holds then $N=\int_{\Omega}(\lambda^*S_0+(1-\lambda^*)r)dx$ for some $\lambda^*\in C(\overline{\Omega})$ satisfying $0< \lambda^*<1$ and $\sigma(d_I,\beta\lambda^*(S_0-r))\le 0$. Hence, we must have  $N\le N^*_{S_0,r}$. So alternative {\rm (ii)} must hold whenever $N>N^*_{S_0,r}$. 
\end{proof}

{ 
\begin{proposition} Let $N^*_{S_0,r}$ be defined as in Theorem 
\ref{theorem_massds}. The following statements hold.
   \begin{itemize}\item[\rm (1)]  It  holds that $N^*_{S_0,r}\ge \int_{\Omega}rdx$.
   \item[\rm (2)] It holds that $N^*_{S_0,r}\leq \int_{\Omega}\max\{S_0,r\}dx$. Hence if $N>\int_{\Omega}\max\{S_0,r\}dx$, then alternative (ii) of Theorem \ref{theorem_massds} holds.
   \item[\rm (3)] The strict inequality $N^*_{S_0,r}< \int_{\Omega}\max\{S_0,r\}dx$ holds if $\|(S_0-r)_+\|_{L^\infty(\Omega)}>0$ (Therefore, in general, the condition $N>N^*_{S_0,r}$ is weaker than $N>\int_{\Omega}\max\{S_0,r\}dx$).
   \end{itemize}
\end{proposition}}
\begin{proof} 
(1) Taking $\lambda^*\equiv 0$ in \eqref{N-critic}, we have the desired result. 

\noindent (2) For any $\lambda^* \in C(\bar\Omega; [0, 1])$, we have $\int_\Omega (\lambda^*S_0+(1-\lambda^*)r)dx\le \max\{S_0, r\}$. By the definition of  $N^*_{S_0,r}$, we have  $N^*_{S_0,r}\leq \int_{\Omega}\max\{S_0,r\}dx$.

\noindent (3) Suppose that $\|(S_0-r)_+\|_{L^\infty(\Omega)}>0$ and we prove  $N^*_{S_0,r}<\int_{\Omega}\max\{S_0,r\}dx$ by contradiction.  Suppose to the contrary that there is $\lambda^*_k\in C(\overline{\Omega};[0,1])$ satisfying ${ \sigma(d_I,\beta\lambda^*_k(S_0-r))}\le 0$ for each $k\ge 1$ such that 
 $$
 \int_{\Omega}(\lambda^*_kS_0+(1-\lambda^*_k)r)dx\to \int_{\Omega}\max\{S_0,r\}dx\quad \text{as}\quad k\to\infty.
 $$
 Since $\|\lambda^*_k\|_{L^{\infty}(\Omega)}\le 1$,  possibly after passing to a subsequence and using the Banach-Alaoglu theorem, there is $\lambda^*\in L^{\infty}(\Omega)$ satisfying $0\le \lambda^*\le 1$ almost everywhere on $\Omega$, such that  $\lambda^*_k\to \lambda^*$ weakly-star in $L^\infty(\Omega)$ as $k\to\infty$. So we have 
 $$
 \int_{\Omega}(\lambda^*_kS_0+(1-\lambda^*_k)r)dx\to \int_{\Omega}(\lambda^*S_0+(1-\lambda^*)r)dx\quad \text{as} \quad k\to\infty. 
 $$
It follows that
 $$
 \int_{\Omega}(\lambda^*S_0+(1-\lambda^*)r)dx=\int_{\Omega}\max\{S_0,r\}dx,
 $$
 which yields that $\max\{S_0,r\}=\lambda^*S_0+(1-\lambda^*)r$ almost everywhere on $\Omega$. So, $\beta\lambda^*(S_0-r)=\beta(\max\{S_0,r\}-r)$ almost everywhere on $\Omega$. Therefore, by the assumption  $\|(S_0-r)_+\|_{\infty}>0$,
 \begin{equation}\label{NN1}
 \int_{\Omega}\beta\lambda^*(S_0-r)dx=\int_{\Omega}\beta(\max\{S_0,r\}-r)dx=\int_{\{S_0>r\}}\beta(S_0-r)dx>0.
 \end{equation}
 However since ${ \sigma(d_I,\beta\lambda^*_k(S_0-r))}\le 0$, we have $\int_{\Omega}\beta\lambda^*_k(S_0-r))dx\le 0$ for any $k\ge 1$ by Lemma \ref{eig-value-lemma}. Letting $k\to\infty$, we get $\int_{\Omega}\beta\lambda^*(S_0-r)dx\le0$, which contradicts with \eqref{NN1}.
\end{proof}

We complement Theorem \ref{theorem_massds} with a corollary: 
\begin{corollary}\label{cor1}
    Suppose that {\rm (A1)-(A2)} holds, $N>\int_{\Omega}rdx$, and $d_I>0$. If either $\beta$ is constant or $S_0-r$ has a constant sign, then  alternative {\rm (ii)} of Theorem \ref{theorem_massds} holds for any solution of \eqref{model-mass0}.
\end{corollary}
\begin{proof} Let $N^*_{S_0,r}$ be defined as in \eqref{N-critic}.  (1) If $\beta$ is constant, then for any $\lambda^*\in C(\overline{\Omega};[0,1])$, $\sigma(d_I,\beta\lambda^*(S_0-r))\le 0$ implies that $\int_{\Omega}\lambda^*(S_0-r)dx\le 0$. In this case,  
$$
\int_{\Omega}(\lambda^*S_0+(1-\lambda^*)r)dx=\int_{\Omega}\lambda^*(S_0-r)dx+\int_{\Omega}rdx\le \int_{\Omega}rdx,\quad \forall\ \lambda^*\in C(\overline{\Omega};[0,1]).
$$
So, we have $N^*_{S_0,r}=\int_{\Omega}rdx$.

 \smallskip

\noindent (2) If $S_0\le r$, then $N^*_{S_0,r}=\int_{\Omega}rdx$. 

 \smallskip

\noindent (3) If $S_0\ge r$, then $N^*_{S_0,r}\le \int_{\Omega}S_0dx$.

 \smallskip

\noindent In cases (1)-(3), we have $N^*_{S_0,r}\le \max\{\int_{\Omega}rdx,\int_{\Omega}S_0dx\}$. By hypothesis (A2), we always have $N>\int_{\Omega} S_0dx$.
Therefore if either of these scenarios holds,   $N>\int_{\Omega}rdx$ implies  $N>N^*_{S_0,r}$. The conclusion now follows from Theorem \ref{theorem_massds}.
\end{proof}

% \begin{remark} Let $N^*_{S_0,r}$ be defined as in \eqref{N-critic}. Then the strict inequality $N^*_{S_0,r}< \int_{\Omega}\max\{S_0,r\}dx$ holds if $\|(S_0-r)_+\|_{\infty}>0$ (see a proof right after the remark).  Moreover, alternative (ii) of Theorem \ref{theorem_massds} always holds when $N>\int_{\Omega}r$ if $N^*_{S_0,r}\le \max\{\int_{\Omega}S_0,\int_{\Omega}r\}$. Sufficient conditions  ensuring the validity of  the latter inequality are given in Corollary \ref{cor1}. It remains an open problem to know whether alternative (i) of Theorem \ref{theorem_massds} may hold for some initial data satisfying $\max\{\int_{\Omega}S_0,\int_{\Omega}r\}<N<N^*_{S_0,r}$. Proposition \ref{Prop-appendix} of the Appendix provides an example of a positive and continuous function $S_0$ satisfying $\max\{\int_{\Omega}S_0,\int_{\Omega}r\}<N^*_{S_0,r}$. Whenever such $S_0$ is fixed, we can always select $I_0$ to be small enough such that $N=\int_{\Omega}(S_0+I_0)$ satisfies $\max\{\int_{\Omega}S_0,\int_{\Omega}r\}<N<N^*_{S_0,r}$. 
% \end{remark}

\begin{remark}  Let $N^*_{S_0,r}$ be defined as in \eqref{N-critic}.  Alternative (ii) of Theorem \ref{theorem_massds} always holds when $N>\int_{\Omega}rdx$ if $N^*_{S_0,r}\le \max\{\int_{\Omega}S_0dx,\int_{\Omega}rdx\}$. Sufficient conditions  ensuring the validity of  the latter inequality are given in Corollary \ref{cor1}. It remains an open problem to know whether alternative (i) of Theorem \ref{theorem_massds} may hold for some initial data satisfying $\max\{\int_{\Omega}S_0dx,\int_{\Omega}rdx\}<N<N^*_{S_0,r}$. { Note that it is possible to construct examples of} positive and continuous functions $S_0$ satisfying $\max\{\int_{\Omega}S_0dx,\int_{\Omega}rdx\}<N^*_{S_0,r}$. Whenever such $S_0$ is fixed, we can always select $I_0$ to be small enough such that $N=\int_{\Omega}(S_0+I_0)dx$ satisfies $\max\{\int_{\Omega}S_0dx,\int_{\Omega}rdx\}<N<N^*_{S_0,r}$.

% It remains an open problem to see, when $\|(S_0-r)_+\|_{\infty}>0$ ({  Salako: We need $S_0-r$ to change sign, since the case of $S_0\ge,\ne r $ is considered in Corollary 3.5. In fact in this case, we always have that $\int_{\Omega}r<N^*_{S_0,r}<\int_{\Omega}S_0$. }) and $\beta$ is not constant,  whether {  the strict inequality $\int_{\Omega}rdx<N^*_{S_0,r}$ can hold and} alternative (i) of Theorem \ref{theorem_massds} may hold for some initial data satisfying $\int_{\Omega}rdx<N<N^*_{S_0,r}.$
% {  Wu: do we need  to prove  $\int_{\Omega}rdx<N^*_{S_0,r}$? Salako: We may not be able to prove this in general (I  may be wrong). But, it seems that  we can construct some examples (See Proposition 6.2).  So, since, there is no open problem when $N^*_{S_0,r}\le \max\{\int_{\Omega}S_0,\int_{\Omega}r\}$, I have modified the open question accordingly. 
\end{remark}

% \begin{remark}
%     Recall from hypothesis (A2) that it always holds that $N>\int_{\Omega}S_0dx$.  Hence  alternative (ii) of Theorem \ref{theorem_massds} always holds whenever $N>\int_{\Omega}rdx$ and  $S_0-r$ does not change sign. Therefore, Theorem \eqref{theorem_massds} indicates that if either $\beta$ is constant or $S_0-r$ does not change sign on $\Omega$, the quantity $\int_{\Omega}rdx$ serves as a threshold value for the total size of the population for disease persistence when the movement of the susceptible population is  completely restricted. 
%     % either $ N>\int_{\Omega}rdx$ and $S_0\le r$ or $N>\int_{\Omega}S_0dx$ (implied by (A2)) and $S_0\ge r$  imply $N>N_{r, S_0}^*$, where $N^*_{r,S_0}$ is defined by \eqref{N-critic}. Moreover, it is always true that $N^*_{r,S_0}\le \int_{\Omega}\max\{r,S_0\}dx$, and hence the assumption $N>\int_{\Omega}\max\{r,S_0\}dx$ also implies $N>N_{r, S_0}^*$.
%     % So alternative (ii) in Theorem \ref{theorem_massds}  holds in these cases.  
% \end{remark}

\subsection{Limiting the movement of infected people}\label{subsection32}
We consider the impact of limiting the movement of infected people on \eqref{model} with  mass action mechanism, that is, the long time behavior of the following degenerate system:
\begin{equation}\label{model-mass1}
\begin{cases}
\partial_t S=d_S\Delta S-\beta(x) S I+\gamma(x)  I, &x\in\Omega, t>0,\\
\partial_t I=\beta(x) S I-\gamma(x)  I,&x\in\bar\Omega, t>0,\\
\partial_\nu S=0, &x\in\partial\Omega,t>0,\\
S(x, 0)=S_0(x), \ I(x, 0)=I_0(x), &x\in\bar\Omega.
\end{cases}
\end{equation}

In the following result, we show the global existence of the solution of \eqref{model-mass1}.  We remark that the $S$ component of the solution is globally bounded while the $I$ component may blow up at $t=\infty$.

\begin{proposition}\label{prop_ex1}
Suppose that {\rm (A1)-(A2)} holds and $d_S>0$. Then \eqref{model-mass1} has a unique nonnegative global solution $(S, I)$, where 
$$S\in C([0, \infty), C(\overline{\Omega}))\cap C^1((0, \infty), {\rm Dom}_{\infty}(\Delta))\quad \text{and}\quad I\in C^1([0, \infty), C(\overline{\Omega})).$$ Moreover, $\|S(\cdot, t)\|_{L^\infty(\Omega)}\le \max\{\|S_0\|_{L^\infty(\Omega)}, \ r_M\}$ for all $t\ge 0$.
\end{proposition}
\begin{proof}
Suppose that a nonnegative solution exists. Let $M_1=\max\{\|S_0\|_{L^\infty(\Omega)}, \ r_M\}$. By the first equation of \eqref{model-mass1} and the comparison principle, $0\le S(x, t)\le M_1$ for all $x\in\bar\Omega$ and $t>0$, which implies that % Let $M_2:=\beta_M M_1+\gamma_M$. By the second equation of \eqref{model-mass1}, 
$$
0\le I(x, t)= I_0(x) e^{\int_0^t (\beta(x) S(x, s)-\gamma(x))ds}\le I_0(x)e^{\beta_MM_1t},  \ \ \forall x\in\bar\Omega, t\ge 0.
$$
This gives $\|I(\cdot, t)\|_{L^\infty(\Omega)}\le \|I_0\|_{L^\infty(\Omega)} e^{\beta_{M}M_1t}$ for all $t\ge 0$. The local existence of the solution of \eqref{model-mass1} can be proved similar to Proposition \ref{prop}. The a prior bound of the solution ensures the global existence of it. 
\end{proof}

To study the asymptotic behavior of the solutions of \eqref{model-mass1}, we  use the following Lyapunov function
$$
V(S, I)=\int_\Omega\left(\frac{1}{2} S^2+rI \right)dx. 
$$
If $(S, I)$ is the solution of \eqref{model-mass1}, it is easy to check that 
\begin{equation}\label{ly1}
\frac{d}{dt}V(S, I)=-d_S\int_\Omega |\triangledown S|^2dx-\int_\Omega \beta (S-r)^2Idx.
\end{equation}

\begin{remark}\label{remark_ly}
We point out that it is very easy to draw a false conclusion using the above Lyapunov function: firstly by the term $\int_\Omega \beta (S-r)^2Idx$ in \eqref{ly1}, one may conclude that either $S\to r$ or $I\to 0$ as $t\to\infty$; then by the term $\int_\Omega |\triangledown S|^2dx$,  $\triangledown S\to 0$ and so $I\to 0$  if $r$ is not constant. We will show that this intuition is indeed false in Theorem \ref{theorem_S}. Actually, it is possible that $S$ converges to some constant $\bar S$ with $r_m\le \bar S\le r_M$ and $I$ converges to some  measure supported at $\{x\in\bar\Omega: r(x)=\bar S\}$.
\end{remark}

To conclude that $\int_\Omega |\triangledown S|^2dx\to 0$ or $\int_\Omega \beta (S-r)^2Idx\to 0$ as $t\to\infty$, we will need the following lemma.

\begin{lemma}\label{New-Lem1}
    Suppose that (A1)-(A2) holds and $d_S>0$. Let $(S,I)$ be the solution of \eqref{model-mass1}. Then the following conclusions hold:
    \begin{itemize}
        \item[\rm (i)] The mapping $[1,\infty)\ni t\mapsto S(\cdot,t)-\overline{S(\cdot,t)}\in {L^1(\Omega)}$ is H\"older continuous. Furthermore if ${  n=1 }$, then the mapping $[1,\infty)\ni t\mapsto S(\cdot,t)-\overline{S(\cdot,t)}\in { {W^{1,2}(\Omega)}}$ is also H\"older continuous.
        \item[\rm(ii)] If ${  n=1 }$, the mapping $[1,\infty)\ni t\mapsto \int_{\Omega}\beta(S-r)^2Idx$ is H\"older continuous.
    \end{itemize}
\end{lemma}
 \begin{proof} 
%Set $\mathcal{Z}:=\{u\in L^1(\Omega)\, :\, \int_{\Omega}u=0\}$. 
%     Recall that the Laplace operator subject to the homogeneous Neumann  boundary  conditions on $\partial\Omega$ generates an analytic $c_0$-semigroup  $\{e^{t\Delta}\}_{t\ge 0}$ on $L^1(\Omega)$ (see \cite{Amann1983}). Moreover, $\{e^{t\Delta}\}_{t\ge 0}$, restricted to $\mathcal{Z}$, is also an analytic $c_0$-semigroup, and there exist $C_0>0$ and $\sigma>0$ (\cite[Lemma 1.3]{winkler2010aggregation}) such that 
%     \begin{equation*}
%         \|e^{t\Delta}z\|_{L^1(\Omega)}\le C_0e^{-\sigma t}\|z\|_{L^1(\Omega)}, \ z\in\mathcal{Z}.
%     \end{equation*} 
    Setting $Z=S(\cdot,t)-\overline{S(\cdot,t)}$ and $F(\cdot,t)=\beta(r-S(\cdot,t))I(\cdot,t)-\overline{\beta(r-S(\cdot,t))I(\cdot,t)}$, it holds that $Z(\cdot,t), F(\cdot,t)\in \mathcal{Z}$ for all $t\ge 0$ and 
    $$
    \begin{cases}
        \partial_tZ=d_S\Delta Z+F(\cdot,t), & t> 0,\ x\in\Omega,\cr 
        \partial_{\nu}Z=0,\ & t>0,\ x\in\partial\Omega.
    \end{cases}
    $$
    Hence by the variation of constant formula, we have
    $$   Z(\cdot,t)=e^{td_S\Delta_{|\mathcal{Z}}}Z(\cdot,0)+\int_{0}^{t}e^{(t-\tau)\Delta_{|\mathcal{Z}}}F(\cdot,\tau)d\tau, \quad \forall\ t>0.
    $$
    By Proposition \ref{prop_ex1} with $M:=\max\{\|S_0\|_{L^{\infty}(\Omega)},r_M\}$, it holds that 
    \begin{equation}\label{KBB3}
        \|F(\cdot,t)\|_{L^1(\Omega)}\le 2\beta_{M}M\|I(\cdot,t)\|_{L^1(\Omega)}\le M_1:=2NM\beta_M,\quad \forall\ t\ge 0.
    \end{equation}
    
\noindent {\rm (i)}       Fix $0\le\tilde{\alpha}< \tilde{\alpha}+\alpha<1$. For any $t\ge 1$ and $h>0$, by \eqref{KBB0}-\eqref{KBB2} and \eqref{KBB3},
    \begin{align}\label{KBB4}
       & \|Z(t+h)-Z(t)\|_{\mathcal{Z}^{\tilde{\alpha}}}\cr 
       \le & \|(e^{hd_S\Delta_{|\mathcal{Z}}}-{\rm id})e^{td_S\Delta_{|\mathcal{Z}}}Z_0\|_{\mathcal{Z}^{\tilde{\alpha}}} +\int_{0}^{t}\|(e^{hd_S\Delta_{|\mathcal{Z}}}-{\rm id})e^{(t-\tau)d_S\Delta_{|\mathcal{Z}}}F(\cdot,\tau)\|_{\mathcal{Z}^{\tilde{\alpha}}}d\tau\cr 
       &+\int_{0}^h\|e^{(h-\tau)d_S\Delta_{|\mathcal{Z}}}F(\cdot,t+\tau)\|_{\mathcal{Z}^{\tilde{\alpha}}}d\tau\cr 
        \le& C_{\alpha}d_S^{\alpha}h^{\alpha}\Big(\|e^{td_S\Delta_{|\mathcal{Z}}}Z_0\|_{\mathcal{Z}^{\tilde{\alpha}+\alpha}}+\int_{0}^t\|e^{(t-\tau)d_S\Delta_{|\mathcal{Z}}}F(\cdot,\tau)\|_{\mathcal{Z}^{\tilde{\alpha}+\alpha}}d\tau\Big)\cr &+C_{\tilde{\alpha}}\int_{0}^h(d_S(h-\tau))^{-
        \tilde{\alpha}}e^{-d_S\sigma(h-\tau)}\|F(\cdot,t+\tau)\|_{L^1(\Omega)}d\tau\cr 
        \le& C_{\alpha}C_{\tilde{\alpha}+\alpha}d_{S}^{-\tilde{\alpha}}h^{\alpha}\Big(t^{-(\tilde{\alpha}+\alpha)}e^{-\sigma td_S}\|Z_0\|_{L^1(\Omega)}+\int_{0}^t(t-\tau)^{-(\tilde{\alpha}+\alpha)}e^{-d_S\sigma(t-\tau)}\|F(\cdot,\tau)\|_{L^1(\Omega)}d\tau\Big)\cr 
        &+C_{\tilde{\alpha}}d_{S}^{-\tilde{\alpha}}M_1\int_{0}^h\tau^{-\tilde{\alpha}}e^{-d_S\sigma\tau}d\tau\cr 
        \le & C_{\alpha}C_{\alpha+\tilde{\alpha}}d_{S}^{-\tilde{\alpha}}h^{\alpha}\Big(\|Z_0\|_{L^1(\Omega)}+M_1\int_{0}^{\infty}\tau^{-(\tilde{\alpha}+\alpha)}e^{-d_S\sigma \tau}d\tau\Big)+C_{\tilde{\alpha}}d_{S}^{-\tilde{\alpha}}M_1\int_{0}^{h}\tau^{-\tilde{\alpha}}d\tau\cr
        \le &{M}_{\alpha,\tilde{\alpha}}(h^{\alpha}+h^{1-\tilde{\alpha}}).
    \end{align}
    In particular, if $\tilde{\alpha}=0$, we get that 
    $$ 
    \|Z(t+h)-Z(t)\|_{L^1(\Omega)}\le {M}_{\alpha,\tilde{\alpha}}(h^{\alpha}+h),\quad \forall\ t\ge 1, \ h>0,
    $$
   which yields the first assertion of (i). On the other, if ${  n=1 }$, choosing $\alpha\in(\frac{3}{4},1)$, it follows from \cite[Theorem 1.6.1]{DanHenry} that $\mathcal{Z}^{\alpha}$ is continuously embedded in $W^{1,2}(\Omega)$. Hence the last assertion of {\rm (i)} also follows from \eqref{KBB4}.

    {\rm (ii)} Suppose ${  n=1 }$. Let $G(t):=\int_{\Omega}\beta(S-r)^2Idx$ for $t\ge 0$. Since ${  n=1 }$, $W^{1,2}(\Omega)$ is compactly embbedded in $ C(\overline{\Omega})$.  Since the mapping $[1,\infty)\ni t\mapsto Z(\cdot,t)\in{W^{1,2}(\Omega)}$ is H\"older continuous by {\rm (i)},  the mapping $[1,\infty)\ni t\mapsto Z(\cdot,t)\in{C(\overline{\Omega})}$ is also H\"older continuous. This together with the fact that 
    $$
    \sup_{t>0}\Big|\frac{d}{dt}\int_{\Omega}S(x,t)dx \Big|=\sup_{t>0}\Big|\int_{\Omega}\beta(S(x,t)-r(x))I(x,t)dx \Big|\le MN\beta_M
    $$
    implies that the mapping $[1,\infty)\ni t\mapsto S(\cdot,t)\in C(\overline{\Omega}) $ is also H\"older continuous. Thus there exist $0<\tau<1$ and $c>0$ such that 
    $$
    \|S(\cdot,t+h)-S(\cdot,t)\|_{L^{\infty}(\Omega)}\le c|h|^{\tau},\quad t\ge 1.
    $$
    So for any $t\ge 1$ and $h>0$, we have
    \begin{align*}
        |G(t+h)-G(t)|\le& \beta_M\|(S(\cdot,t+h)-r)^2-(S(\cdot,t)-r)^2\|_{\infty}\int_{\Omega}I(x,t+h)dx\cr &+\beta_M\|(S(\cdot,t)-r)^2\|_{L^{\infty}(\Omega)}\int_{\Omega}|I(x,t+h)-I(x,t)|dx\cr
        \le& 2cMNh^{\tau}\beta_M+M^2\beta_M\int_{\Omega}\int_{0}^h\beta|S(x,t+s)-r(x)|I(x,t+s)dsdx\cr
        \le & 2cMNh^{\tau}\beta_M+M^3\beta_M^2\int_{\Omega}\int_{0}^hI(x,t+s)dsdx\cr
        =&2cMNh^{\tau}\beta_M+M^3\beta_M^2\int_{0}^h\int_{\Omega}I(x,t+s)dxds\cr
        \le & 2cMNh^{\tau}\beta_M+M^3\beta_M^2Nh\le (c+M^2\beta_M)M_1(h^{\tau}+h),
    \end{align*}
    which yields the desired result.
\end{proof}

\begin{lemma}\label{lemma_c}
Suppose that ${  n=1 }$, {\rm (A1)-(A2)} holds and $d_S>0$. Let $(S, I)$ be the solution of \eqref{model-mass1}. Then, $\|\triangledown S(\cdot, t)\|_{L^2(\Omega)}\to 0$ and $\int_\Omega \beta (S(x, t)-r)^2I(x, t) dx\to 0$ as $t\to\infty$.
\end{lemma}
\begin{proof}
Integrating \eqref{ly1} over $(0, \infty)$ and by $\int_\Omega I dx\le N$ and Proposition \ref{prop_ex1}, we have 
\begin{equation}\label{bd}
\int_0^\infty \|\triangledown S(\cdot, t)\|_{L^2(\Omega)}^2 dt<\infty \ \text{and} \ \int_0^\infty\int_\Omega \beta (S(x, t)-r)^2I(x, t) dxdt<\infty.
\end{equation}
Then the claim follows from Lemmas \ref{lemma_inf0} and \ref{New-Lem1}.
\end{proof}

We are ready to state the main result concerning the global dynamics of \eqref{model-mass1}. Motivated by the biological meaning and  expression of $\mathcal{R}^1_0$, as in \cite{Allen,WuZou}, we call $H^+$ and $H^-$ as the  high-risk and low-risk sites, respectively, where
$$
H^+=\left\{x\in\bar\Omega: \ \frac{N}{|\Omega|}\beta(x)-\gamma(x)>0 \right\}\ \ 
\text{and}\ \ 
H^-=\left\{x\in\bar\Omega: \ \frac{N}{|\Omega|}\beta(x)-\gamma(x)<0 \right\}.
$$

Define $\tilde{r}_m=\inf_{x\in\{I_0>0\}}r(x)$
and $\mathcal{M}:=\{x\in\overline{\{I_0>0\}}\ :\ r(x)=\tilde{r}_m \}$. Biologically, $\mathcal{M}$ consists with all the points of the highest risk {relative to $I_0$}. We will show that the infected people will concentrate on $\mathcal{M}$ when limiting their movement.

\begin{theorem}\label{theorem_S}
Suppose that  {\rm (A1)-(A2)} holds and $d_S>0$. Let $(S, I)$ be the solution of \eqref{model-mass1}.  Then, we have
\begin{equation}\label{theorem_S-eq}
    \lim_{t\to\infty}\|S(\cdot,t)-\overline{S(\cdot,t)}\|_{L^{p}(\Omega)}=0, \quad \forall\   p\in[1,\infty).
\end{equation}
If in addition ${  n=1 }$, then the limit in \eqref{theorem_S-eq} also holds for $p=\infty$, 
\begin{equation}\label{KI}
\lim_{t\to\infty}\|I(\cdot,t)\|_{L^{\infty}(K\cup \{I_0=0\})}= 0
\end{equation}
for any compact set $K\subset H^{-}$ if $H^-$ is not empty, and the following conclusions hold:
\begin{itemize}
\item[\rm (i)] If $H^+\cap \{I_0>0\}=\emptyset$, then  $\|S(\cdot,t)-N/{|\Omega|}\|_{L^\infty(\Omega)}\to 0$ and $\int_{\Omega}I(x,t)dx\to 0$ as $t\to\infty$;

\item[\rm (ii)]
If $H^+\cap \{I_0>0\}\ne\emptyset$, then there is a sequence $\{t_k\}_{k\ge 1}$ converging to infinity such that $\|S(\cdot,t_k)-\tilde{r}_m\|_{L^\infty(\Omega)}\to 0$, $\int_{\Omega}I(x,t_k)dx\to N-|\Omega|\tilde{r}_m$ and $I(\cdot,t_k)\to 0 $ as $k\to\infty$ almost everywhere on $\{ x\in\Omega :\ r(x)\ne \tilde{r}_m\}$. In particular,  if  $\mathcal{M}=\{x_1,\cdots,x_L\}\subset\{I_0>0\}$, then $ I(\cdot,t_k)\to (N-|\Omega|\tilde{r}_m)\sum_{i=1}^Lc_i\delta_{x_i}$ weakly as $k\to\infty$, where $0\le c_i\le1$, $\sum_{i=1}^Lc_i=1$, and  $\delta_{x_i}$ is the Dirac measure centered at $x_i$.
\end{itemize}
\end{theorem}
\begin{proof}
By the Poincar\'e inequality, there is a positive constant $\lambda_0>0$ such that 
\begin{equation}\label{KBB6}
\|S(\cdot,t)-\overline{S(\cdot,t)}\|_{L^2(\Omega)}\le \lambda_0\|\nabla S(\cdot,t)\|_{L^2(\Omega)},\quad \forall\ t>0.
\end{equation}
 Hence using  H\"older's inequality and recalling \eqref{ly1}, we get that 
\begin{align*}
\int_0^{\infty}\|S(\cdot,t)-\overline{S(\cdot,t)}\|_{L^1(\Omega)}^2dt\le & |\Omega|\int_{0}^{\infty}\|S(\cdot, t)-\overline{S(\cdot,t)}\|_{L^2(\Omega)}^2dt\cr 
\le&  \lambda_0^2|\Omega|\int_{0}^{\infty}\|\nabla S(\cdot, t)\|_{L^2(\Omega)}^2dt<\infty.
\end{align*}
Therefore by Lemmas \ref{lemma_inf0} and \ref{New-Lem1},  we obtain  $\|S(\cdot, t)-\overline{S(\cdot,t)}\|_{L^1(\Omega)}^2\to 0$ as $t\to\infty$. By $\sup_{t\ge 1}\|S(\cdot,t)\|_{\infty}<\infty$ and H\"older's inequality,  \eqref{theorem_S-eq} holds.

From this point, we shall suppose that ${  n=1 }$ and complete the proof of the theorem. In view of Lemma \ref{lemma_c} and inequality \eqref{KBB6}, we have  
$$
\lim_{t\to\infty}\|S(\cdot,t)-\overline{S(\cdot,t)}\|_{W^{1,2}(\Omega)}=0.
$$
Since ${  n=1 }$, $W^{1,2}(\Omega)$ is compactly embedded into $C(\bar{\Omega})$.
Therefore,  
\begin{equation}\label{BS1}
    \lim_{t\to\infty}\big\| S(\cdot,t)-\overline{S(\cdot,t)}\big\|_{L^\infty(\Omega)}=0.
\end{equation}
Fix $\varepsilon>0$. By $N=\int_{\Omega}(S+I)dx$ for all $t\ge 0$ and \eqref{BS1}, there is $t_{\varepsilon}>0$ such that 
\begin{equation}\label{BS1-2}
S(x,t)\le \frac{1}{|\Omega|}\Big(N+\varepsilon-\int_{\Omega}I(x,t)dx\Big),\quad \forall\ (x, t)\in \bar\Omega\times [t_{\varepsilon}, \infty).
\end{equation}
Let $K\subset H^{-}$ be a compact set if $H^-$ is not empty. By the definition of $H^-$, we have $\min_{x\in K}r(x)>N/|\Omega|$. If $0<\varepsilon\ll 1$ is chosen such that $\eta_{\varepsilon}:=(N+\varepsilon)/|\Omega|-\min_{x\in K}r(x)<0$, then
$$ 
\partial_tI(x,t)\le \beta(x)\Big(\frac{N+\varepsilon}{|\Omega|}-\min_{x\in K}r(x)\Big)I(x,t)\le \beta_m\eta_{\varepsilon}I(x,t),  \quad (x, t)\in K\times [t_{\varepsilon}, \infty).
$$
It follows that
$$
\|I(\cdot,t)\|_{L^{\infty}(K)}\le e^{(t-t_\varepsilon)\beta_m\eta_{\varepsilon}}\|I(\cdot,t_{\varepsilon})\|_{L^{\infty}(K)}\to 0,\quad \text{as}\ t\to\infty.
$$
This together with the fact that $I(x,t)=0$ for all $t\ge 0$ and $x\in\{I_0=0\}$ completes the proof of \eqref{KI}.

To prove {\rm (i)}-{\rm (ii)}, we claim that
\begin{equation}\label{BS2}   \limsup_{t\to\infty}\|I(\cdot,t)\|_{L^1(\Omega)}\le (N-|\Omega|\tilde{r}_{m})_+.
\end{equation}
Since $\int_{\{I_0=0\}}I(x,t)dx=0$ for all $t\ge 0$, it suffices to show $\limsup_{t\to\infty}\|I(\cdot,t)\|_{L^1(\{I_0>0\})}\le (N-|\Omega|\tilde{r}_{m})_+$.
To see this, observe from \eqref{BS1-2} that 
\begin{equation}\label{BS3}
    \partial_tI\le \frac{\beta}{|\Omega|}\Big( (N-|\Omega|\tilde{r}_{m})_++\varepsilon-\|I(\cdot,t)\|_{L^1(\{I_0>0\})}\Big)I,\quad \forall\ t\ge t_{\varepsilon},\ x\in \{I_0>0\},
\end{equation}
where $\tilde{r}_m=\inf_{x\in\{I_0>0\}}r(x)$. Let 
\begin{equation}\label{F-new-eq}
F(t):=\frac{\int_{\{I_0>0\}}\beta I(x,t)dx}{\int_{\{I_0>0\}}I(x,t)dx}, \ \ \forall \ t\ge 1.
\end{equation} 
Then $F(t): [1, \infty)\to \mathbb{R}_+$  is Locally Lipschitz continuous with $\beta_m\le F(t)\le \beta_M$, $t\ge 1$. Integrating \eqref{BS3} over $\{I_0>0\}$, for any $t>t_{\varepsilon}$, we get 
\begin{align*}\label{BS5}
    \frac{d}{dt}\|I\|_{L^1(\{I_0>0\})}\le& \frac{1}{|\Omega|} \Big((N-|\Omega|\tilde{r}_m)_++\varepsilon-\|I\|_{L^1(\{I_0>0\})}\Big)\int_{\{I_0>0\}}\beta Idx\cr 
    =& \frac{F(t)}{|\Omega|}\Big((N-|\Omega|\tilde{r}_{m})_++\varepsilon-\|I\|_{L^1(\{I_0>0\})}\Big)\|I\|_{L^1(\{I_0>0\})},
\end{align*}
where $F$ is defined by \eqref{F-new-eq}. Let $v(t)$ be the solution of 
\begin{equation*}
\left\{
\begin{array}{lll}
 v'(t)&=&  \ds\frac{F(t)}{|\Omega|}\Big((N-|\Omega|\tilde{r}_{m})_++\varepsilon-v(t)\Big)v(t),\ \ t>t_\varepsilon,\\
 v(t_\varepsilon)&=& \|I(\cdot, t_\varepsilon)\|_{L^1(\{I_0>0\})}.
\end{array}
\right.
\end{equation*}
By the comparison principle, we know that
$$
\|I(\cdot,t)\|_{L^1(\{I_0>0\})}\le v(t),\quad \forall t\ge t_{\varepsilon}.
$$
Since  $v(t)\to (N-|\Omega|\tilde{r}_{m})_++\varepsilon$ as $t\to\infty$ (because $F\ge \beta_m>0$ and $v(t_\varepsilon)>0$)  and $\varepsilon$ is { arbitrarily} chosen, \eqref{BS2} holds.

{\rm (i)} Suppose that $H^+\cap\{I_0>0\}=\emptyset$. Then we have $(N-|\Omega|\tilde{r}_{m})_+=0$. By \eqref{BS2},  $\|I(\cdot,t)\|_{L^1(\Omega)}\to 0$ as $t\to\infty$. This together with the fact that $\int_{\Omega}Sdx=N-\int_{\Omega}Idx$ for all $t>0$, yields $\int_{\Omega}S(x,t)dx\to N$ as $t\to\infty$. It then follows from \eqref{BS1}  that $\|S(\cdot,t)-{N}/{|\Omega|}\|_{L^\infty(\Omega)}\to 0$ as $t\to\infty$.

{\rm (ii)} Suppose that $H^+\cap \{I_0>0\}\ne \emptyset$.  Then we have $N/|\Omega|>\tilde{r}_m$. Since $\int_{\Omega}Sdx=N-\int_{\Omega}Idx$, we conclude from \eqref{BS2} that 
$$
\liminf_{t\to\infty}\int_{\Omega}S(x,t)dx\ge|\Omega|\tilde{r}_{\min},
$$which in view of \eqref{BS1} implies that 
\begin{equation}\label{BS6}
    \liminf_{t\to\infty}\min_{x\in\overline{\Omega}}S(x,t)\ge \tilde{r}_m.
\end{equation}
Now, we claim that 
\begin{equation}\label{BS7}
    \liminf_{t\to\infty}\min_{x\in\overline{\Omega}}S(x,t)=\tilde{r}_m.
\end{equation}
We proceed by contradiction. Suppose to the contrary that \eqref{BS7} is false. Thanks to \eqref{BS6}, there exist $0<\tilde{\varepsilon}\ll 1$ and $\tilde{t}_0>0$ such that 
$$
S(x,t)\ge \tilde{r}_{m}+\tilde{\varepsilon},\quad \forall\ t\ge \tilde{t}_0.
$$
Since $r$ is continuous on $\bar{\Omega}$ and $\tilde{r}_m=\inf_{x\in\{I_0>0\}}r(x)$, there is an open set $\mathcal{O}\subset\{I_0>0\}$ such that $r(x)<\tilde{r}_{m}+\tilde{\varepsilon}/{2}$ for all $x\in\mathcal{O}$. Hence, we have
$$
\partial_{t}I(x,t)=\beta (S-r)I\ge \frac{\tilde{\varepsilon}}{2}\beta_mI(x,t),\quad t>\tilde{t}_0,\ x\in\mathcal{O}.
$$
An integration of this inequality yields 
$$
N\ge \int_{\mathcal{O}}I(x,t)dx\ge e^{\frac{\tilde{\varepsilon}}{2}\beta_m(t-\tilde t_0)}\int_{\mathcal{O}}I(x,\tilde{t}_0)dx,\quad t>\tilde{t}_0.
$$
This is clearly impossible since $ \int_{\mathcal{O}}I(x,\tilde{t}_0)dx>0$. Therefore, \eqref{BS7} must hold.

By \eqref{BS7}, there is a sequence $\{t_k\}_{k\ge 1}$ converging to infinity such that $\min_{x\in\overline{\Omega}}S(x,t_k)\to \tilde{r}_m$ as $k\to\infty$. {  Hence, since 
\begin{align*}
\|\tilde{r}_m-S(\cdot,t_k)\|_{L^\infty(\Omega)}\le& |\tilde{r}_m-\overline{S(\cdot,t_k)}|+ \|S(\cdot,t_k)-\overline{S(\cdot,t_k)}\|_{L^\infty(\Omega)} \cr
\le & |\tilde{r}_m-\min_{x\in\overline{\Omega}}S(x,t_k)|+|\min_{x\in\overline{\Omega}}S(x,t_k)-\overline{S(\cdot,t_k)}|+\|S(\cdot,t_k)-\overline{S(\cdot,t_k)}\|_{L^\infty(\Omega)} \cr 
\le &  |\tilde{r}_m-\min_{x\in\overline{\Omega}}S(x,t_k)|+2\|S(\cdot,t_k)-\overline{S(\cdot,t_k)}\|_{L^{\infty}(\Omega)},\quad \forall\ k\ge 1,
\end{align*}
we conclude from \eqref{BS1} that 
} %Hence by \eqref{BS1},
$\|S(\cdot,t_k)-\tilde{r}_m\|_{L^\infty(\Omega)}\to 0$ as $k\to\infty$. This in turn implies that  $\int_{\Omega}I(x,t_k)dx\to N-|\Omega|\tilde{r}_m$ as $k\to\infty$. However, by Lemma \ref{lemma_c}, we know that $\int_{\{I_0>0\}}(S(x,t_k)-r)^2I(x,t_k)dx=\int_{\Omega}(S(x,t_k)-r)^2I(x,t_k)dx\to 0$ as $k\to\infty$. Therefore, possibly after passing to a subsequence, $I(\cdot,t_k)\to 0 $ as $k\to\infty$ almost everywhere on $\{ x\in\Omega :\ r(x)\ne \tilde{r}_m\}$.
Finally since $\{I(\cdot, t_k)\}$ is bounded in $L^1(\Omega)$ and by Riesz representation theorem, passing to a subsequence if necessary, $I(\cdot, t_k)\to (N-|\Omega|\tilde{r}_m)\mu$ weakly as $k\to\infty$ for some probability Radon measure $\mu$.   Since
$$
 \int_{\{I_0>0\}} \beta (S(x, t_k)-r)^2I(x, t_k) dx\to (N-|\Omega|\tilde{r}_m)\int_{\{I_0>0\}} \beta (r_m-r)^2d\mu=0\ \  as \ {k\to\infty},
 $$
 $\mu$ is supported in $\overline{\{I_0>0\}}\cap \mathcal{M}$. In particular if  $\mathcal{M}=\{x_1,\cdots,x_L\}\subset\{I_0>0\}$, then $\mu=\sum_{i=1}^Lc_i\delta_{x_i}$ for some $0\le c_i\le1$ with $\sum_{i=1}^Lc_i=1$.
 % {  Wu: what happened if $\mathcal{M}\subset \partial \{I_0>0\}$? Salako: Great question!!! }
\end{proof}

\begin{remark}
We conjecture that Theorem \ref{theorem_S} holds for any { $n\ge 1$} and $\|S(\cdot,t)-\tilde{r}_m\|_{L^\infty(\Omega)}\to 0$ as $t\to\infty$ in {\rm (ii)}.
\end{remark}

\section{Model with standard incidence mechanism}
\subsection{Limiting the movement of susceptible people}
First, we consider the impact of limiting the movement of susceptible people on \eqref{model} with standard incidence mechanism by setting $d_S=0$, i.e.,
\begin{equation}\label{model-incidence0}
\begin{cases}
\displaystyle\partial_t S=-\beta(x) \frac{S I}{S+I}+\gamma(x)  I, &x\in\bar\Omega, t>0,\\
\displaystyle\partial_t I=d_I\Delta I+\beta(x) \frac{S I}{S+I}-\gamma(x)  I,&x\in\Omega, t>0,\\
\partial_\nu I=0, &x\in\partial\Omega,t>0,\\
S(x, 0)=S_0(x), \ I(x, 0)=I_0(x), &x\in\bar\Omega.
\end{cases}
\end{equation}

\begin{proposition}\label{prop-incidence}
Suppose that {\rm (A1)-(A2)} holds and $d_I>0$. Then \eqref{model-incidence0} has a unique nonnegative global solution $(S, I)$, where 
$$S\in C^1([0, \infty), C(\overline{\Omega}))\quad \text{and}\quad I\in C([0, \infty), C(\overline{\Omega}))\cap C^1((0, \infty), {\rm Dom}_{\infty}(\Delta)).$$ 
Moreover, there exists $M>0$ depending on (the $L^\infty$ norm of) initial data such that 
\begin{equation}\label{bound-4}
  \|I(\cdot, t)\|_{L^\infty(\Omega)}\le M, \ \forall  t\ge 0.
\end{equation}
\end{proposition}
\begin{proof}
If we define $SI/(S+I)=0$ when $(S, I)=(0, 0)$, then $SI/(S+I)$ is Lipschitz in the first quadrant. So the existence and uniqueness of nonnegative local solution $(S, I)$ can be proved using  the Banach fixed point theorem, where  $S\in C^1([0, T), C(\overline{\Omega}))$ and $I\in C([0, T), C(\overline{\Omega}))\cap C^1((0, T), {\rm Dom}_\infty(\Delta))$ for some $T>0$.  

Since the right hand side of the second equation of \eqref{model-incidence0} has linear growth rate (i.e. $\beta SI/(S+I)-
\gamma I\le C I$ for some constant $C>0$) and $\int_\Omega I dx\le N$ for all $t>0$, by \cite[Theorem 3.1]{alikakos1979application},   there exists $M>0$ such that $\|I(\cdot, t)\|_{L^\infty(\Omega)}\le M$ for all $t>0$.  By the first equation of \eqref{model-incidence0}, we  have $\partial_t S\le \gamma I$, which implies 
$$
\|S(\cdot, t)\|_{L^\infty(\Omega)}\le \|S_0\|_{L^\infty(\Omega)}+M\gamma_Mt, \ \ t\ge 0.
$$
Hence, we can extend the solution globally.
\end{proof}

We define $H^+$, $H^0$, and $H^-$ as the  high-risk, moderate-risk, and low-risk sites, respectively, where $H^+=\left\{x\in\bar\Omega: \ \beta(x)-\gamma(x)>0 \right\}$, $H^0=\left\{x\in\bar\Omega: \ \beta(x)-\gamma(x)=0 \right\}$ and $H^-=\left\{x\in\bar\Omega: \ \beta(x)-\gamma(x)<0 \right\}$. The following result has appeared in the literature.

\begin{theorem} \cite[Theorem 2.5, Lemma 5.6]{LouSalako2021}\label{theorem_SL}
Suppose that {\rm (A1)-(A2)} holds and $d_I>0$. Let $(S, I)$ be the solution of \eqref{model-incidence0}. Then the following conclusions hold:
\begin{enumerate}
\item[\rm (i)] If $H^-$ is nonempty, then $\lim_{t\to \infty}\|S(\cdot, t)-S^*\|_{L^\infty(\Omega)} =0$ and $\lim_{t\to \infty}\|I(\cdot, t)\|_{L^\infty(\Omega)} =0$, where $S^*\in C(\bar\Omega)$ satisfies $S^*>0$  on $H^-\cup H^0$.

\item[\rm (ii)]  If $\beta(x)>\gamma(x)$ for all $x\in\bar\Omega$, then  $\lim_{t\to\infty}\|S(t,\cdot)-\gamma I^*/(\beta-\gamma)\|_{L^\infty(\Omega)}= 0$ and $\lim_{t\to\infty} \|I(t,\cdot)-I^*\|_{L^\infty(\Omega)}= 0$, where $I^*$ is a constant given by 
\begin{equation}\label{I-star}
I^*:=\frac{N}{\int_{\Omega}\frac{\beta}{\beta-\gamma}dx}.
\end{equation}
\end{enumerate}
\end{theorem}

\begin{remark}\label{Remark1}
    The only case not covered by Theorem \ref{theorem_SL} is when $\beta\ge \gamma$ and $H^0$ is not empty, which we will deal with later in this section. In the case $H^-\ne\emptyset$ and $\mathcal{R}_0>1$, it has been shown in   \cite[Theorem 2.5]{LouSalako2021} that $J^*:=\{x\in\bar\Omega: S^*(x)=0\}$ is a subset of $H^+$ such that both $J^*$ and $\Omega\backslash J^*$ have positive measure, where $\mathcal{R}_0$,  defined as 
    \begin{equation}\label{R-0-def}
    \mathcal{R}_0=\sup\left\{\frac{\int_\Omega \beta \phi^2dx}{\int_\Omega (d_I|\triangledown \phi|^2+\gamma \phi^2) dx}:\ \ \phi\in H^1(\Omega)\backslash\{0\}\right\},
    \end{equation}
     is the  basic reproduction number of the  diffusive epidemic model \eqref{model} with standard infection incidence mechanism and  diffusion rates $d_S>0$ and $d_I>0$. 
\end{remark}

\begin{theorem}\label{theorem_stdDS}
Suppose that {\rm (A1)-(A2)} holds and $d_I>0$. Let $(S, I)$ be the solution of \eqref{model-incidence0}. If $\beta(x)\ge\gamma(x)$ for all $x\in\bar\Omega$ and $H^0$ is nontrivial, then the following conclusions hold:
\begin{enumerate}

\item[\rm (i)] If $H^0$ has positive measure, then  there exists $M>0$ depending on (the $L^\infty$ norm of) initial data such that 
\begin{equation}\label{bound-incidence}
  \|S(\cdot, t)\|_{L^\infty(\Omega)}\le M, \ \ \text{for all }  t\ge 0.
\end{equation}
Moreover, there is {  $S^*\in L^{\infty}(\Omega)$ satisfying } $S^*_{|H^0}\in C(H^0)$ and $S^*_{|H^0}>0$ on  $H^0$ such that $\lim_{t\to\infty}\|S(\cdot, t)-S^*\|_{L^{\infty}(H^0)}=0$, { $\lim_{t\to\infty}\|S(\cdot,t)-S^*\|_{L^1(H^+)}=0$, }  and $\lim_{t\to\infty}\|I(\cdot, t)\|_{L^\infty(\Omega)}=0$. {  In addition, if $H^+\ne\emptyset$, then $\text{meas}(\{x\in H^+\, | \, S^*(x)=0\})>0$. }

\item[\rm (ii)]   If $H^0$ has zero measure
and $\int_\Omega 1/(\beta-\gamma)dx=\infty$, then there exists $\{t_k\}$ converging to infinity such that  
 $I(\cdot, t_k)\to 0$ in $C(\bar\Omega)$ {  and $\int_{\Omega}S(\cdot,t_k)dx\to N$} as $k\to\infty$.
%where $S^*\in L^2(\Omega)$, $S^*\ge 0$ on $\Omega$ and $\int_\Omega S^*dx=N$.}
% $\lim_{t\to\infty}\|I(\cdot, t)-I^*\|_{L^\infty(\Omega)}=0$, where $I^*$ is a nonnegative constant with
% $$
% I^*=0 \ \ \text{or}\ \ 0<I^*\le\frac{N}{N+\int_\Omega \frac{\gamma}{\beta-\gamma}dx}.
% $$
% Moreover if $I^*>0$, then $S(x, t)\to \gamma I^*/(\beta-\gamma)$ pointwise for $x\in\Omega\backslash H^0$ and $\int_\Omega 1/(\beta-\gamma)dx<\infty$; if
% $\int_\Omega 1/(\beta-\gamma)dx=\infty$, then $I^*=0$.
\end{enumerate}
\end{theorem}
\begin{proof} 
%(1) The fact that $I(t,\cdot)\to 0$ and $S(t,\cdot)\to S^*$ as $t\to\infty$, uniformly on $\Omega$, where $S^*\in C(\overline{\Omega})$ follows from \cite[Theorem 2.5]{LouSalako2021}. Now, it is easy to see that for every $x\in \Omega^-\cup \Omega^0$, $S(x,t)$ is strictly increasing in $t$. Hence $S^*(x)>S(x)$ for all $x\in\Omega^{-}\cup \Omega^{0}$.

{\rm (i)} By Lemma \ref{lemma_harnack}, there exists $C>1$ such that  
\begin{equation}\label{minmax}
\max_{x\in\bar\Omega} I(x, t)\le C \min_{x\in\bar\Omega} I(x, t), \ \ x\in\bar\Omega, t\ge 1.
\end{equation}
Fix $x_0\in H^0$. Let $x\in\bar\Omega$. Then $\partial_t S=\frac{(-\beta+\gamma)S+\gamma I}{S+I}I\le \gamma I^2/(S+I)$. We consider the following problem:
\begin{equation}\label{model-comp}
\begin{cases}
\displaystyle \bar S'=\gamma(x) \frac{I^2(x, t)}{\bar S+I(x, t)},&t>1,\\
\bar S(1)=S(x, 1). &
\end{cases}
\end{equation}
Then we have $S(x, t)\le \bar S(t)$ for all $t\ge 1$. We claim that $KS(x_0, t)$ is an upper solution of \eqref{model-comp} if $K>0$ is large enough.  To see it, it suffices to check
$$
K\partial_tS(x_0, t)=K\gamma(x_0) \frac{I^2(x_0, t)}{S(x_0, t)+I(x_0, t)}\ge  \gamma(x) \frac{I^2(x, t)}{KS(x_0, t)+I(x, t)}.
$$
Noticing \eqref{minmax}, we only need to check 
$$
K\gamma(x_0) \frac{I^2(x, t)/C^2}{S(x_0, t)+{\frac{I(x,t)}{C}}}\ge  \gamma(x) \frac{I^2(x, t)}{KS(x_0, t)+I(x, t)},
$$
which is equivalent to $(K^2\gamma(x_0)-C^2\gamma(x))S(x_0, t)+(K\gamma(x_0)-{C}\gamma(x))I(x, t)\ge 0$. So we can choose $K$ large independent of $x\in\bar\Omega$ such that the inequality holds. Hence, we have $KS(x_0, t)\ge \bar S(t)\ge S(x, t)$ for all $t\ge 1$. Moreover, interchanging the role of $x_0$ and $x$, we have 
\begin{equation}\label{xx0}
S(x_0, t)/K\le S(x, t)\le KS(x_0, t), \ \ \ \forall x\in H^0, \ t\ge 1. 
\end{equation}
By \eqref{xx0},  $N\ge \int_\Omega S(x, t)dx\ge \int_{H^0} S(x, t)dx\ge \int_{H^0} S(x_0, t)/Kdx=|H^0| S(x_0, t)/K$ for all $t\ge 1$. Therefore, we have $S(x_0, t)\le KN/|H^0|$ and $S(x, t)\le K^2N/|H^0|$ for all $x\in\bar\Omega$ and $t\ge 1$.

The convergence of $(S, I)$ can be proved similar to  \cite[Lemma 5.6]{LouSalako2021}, and we include it for completeness. By Proposition \ref{prop-incidence}, we have $0\le S(x, t), I(x, t)\le M$ for all $x\in\bar\Omega$ and $t\ge 0$. It follows from the equation of $S$  that 
$$
\partial_t S=\frac{\gamma I^2}{S+I}\ge  \frac{\gamma_m}{2M} \min_{ y\in H^0} I^2(y, t), \ \ \forall\ x\in H^0, {  t>0}.
$$
Integrating the above inequality over $H^0\times (0, \infty)$ and noticing that $H^0$ has positive measure, we see that  $\int_0^\infty \min_{x\in H^0} I^2(x, t) dx<\infty$. By \eqref{minmax},
it holds that 
\begin{equation}\label{ibb}
\int_{0}^{\infty}\|I(\cdot,t)\|^2_{L^\infty(\Omega)}dt<\infty.
\end{equation}
By \eqref{bound-4} and the $L^p$ estimate, $I$ is  H\"older continuous on $\bar\Omega\times [1, \infty)$. Therefore, Lemma \ref{lemma_inf0} and \eqref{ibb} imply that $I(x, t)\to 0$ uniformly on $\bar\Omega$ as $t\to\infty$.  

Since $\partial_t S=\gamma I^2/(S+I)$,  $S(x,t)$ is strictly increasing in $t\in(0,\infty)$ for every $x\in H^{0}$ and  
\begin{align}\label{RS5}
\int_{1}^{\infty}\|S_t(\cdot,t)\|_{L^{\infty}(H^0)}dt\le &\gamma_M\int_{1}^{\infty}\frac{\|I(\cdot,t)\|_{L^\infty(\Omega)}^2}{\min_{x\in H^0}S(x,t)}dt\cr 
\le & \frac{\gamma_M}{\min_{x\in H^0}S(x,1)}\int_{1}^{\infty}\|I(\cdot,t)\|^2_{L^\infty(\Omega)}dt<\infty.
\end{align}
Whence, $S(\cdot,t)\to S^*_{|H^0}:=S(\cdot,0)+\int_0^{\infty}S_t(\cdot,t)dt\in C(H^0)$ uniformly on $H^0$ as $t\to \infty$.

{  Next, we discuss the convergence of $S(x,t)$ as $t\to\infty$ for $x\in H^+$. To this end, let $\kappa:=(\beta-\gamma)/\beta$ and define $$
V(S, I)=\frac{1}{2}\int_\Omega\left(\kappa S^2+I^2 \right)dx. 
$$
It is easy to check that 
\begin{equation}\label{ly1st}
\frac{d}{dt}V(S, I)=-d_I\int_\Omega |\triangledown I|^2dx-\int_\Omega \gamma \frac{(\kappa S-I)^2}{S+I}Idx.
\end{equation}
Integrating \eqref{ly1st} over $(0,t)$ and taking $t\to\infty$, we obtain 
\begin{equation}\label{Revise-eq1}
    \int_0^{\infty}\int_{\Omega}\gamma\frac{(\kappa S-I)^2}{S+I}Idxdt<\infty.
\end{equation}
On the other hand, we have 
\begin{align*}
    \frac{1}{2}(\kappa S^2)_t=\gamma\frac{(I-\kappa S)\kappa S I}{S+I}=\gamma\frac{(I-\kappa S)(\kappa S-I+I)I}{S+I}=-\gamma\frac{(I-\kappa S)^2I}{S+I}+\gamma\frac{(I-\kappa S)}{S+I}I^2.
\end{align*}
Hence, by \eqref{ibb} and \eqref{Revise-eq1}, we have that  
\begin{align*}
    \int_{1}^{\infty}\Big\|\frac{(\kappa S^2)_t}{2}\Big\|_{L^{1}(H^+)}dt\le & \int_1^{\infty}\int_{\Omega}\gamma\frac{(I-\kappa S)^2I}{S+I}dxdt+\int_{1}^{\infty}\int_{\Omega}\gamma\frac{|I-\kappa S|}{S+I}I^2dxdt\cr
    \le & \int_1^{\infty}\int_{\Omega}\gamma\frac{(I-\kappa S)^2I}{S+I}dxdt+\gamma_M(1+\|\kappa\|_{\infty})|\Omega|\int_{1}^{\infty}\|I\|_{L^{\infty}(\Omega)}^2dt\cr 
    <&\infty.
\end{align*}
Therefore, there is a  measurable subset $\mathcal{N}\subset H^+$ with ${\text{meas}(\mathcal{N})}=0$  such that 
$$
\int_{1}^{\infty}\Big|\frac{(\kappa S^2)_t}{2}\Big|dt<\infty, \quad \forall\ x\in H^{+}\setminus\mathcal{N}.
$$
As a result, we have that $ \kappa(x) S^2(x,t)\to \kappa(x)(S_0^2(x)+\int_0^{\infty}(S^2)_tdt)$ as $t\to\infty$ for $x\in H^{+}\setminus\mathcal{N}$. So, $S(x,t)\to S^*_{H^+}(x):=\sqrt{S^2_0(x)+\int_0^{\infty}(S^2)_tdt}$ as $t\to\infty$ for $x\in H^+\setminus \mathcal{N}$. Since $\|S(\cdot,t)\|_{L^{\infty}(\Omega)}\le M$ for all $t\ge 1$, then $S^*_{|H^+}\in L^{\infty}(H^+)$, and by the Lebesgue dominated theorem and the fact that $\text{meas}(\mathcal{N})=0$, $\|S(\cdot,t)-S^*_{|H^+}\|_{L^1(H^+)}\to 0$ as $t\to\infty$. Now, taking $S^*:=S^*_{|H^0}\chi_{H^0}+S^*_{|H^{+}}\chi_{H^{+}}$, then $S^*\in L^{\infty}(\Omega)$, $S^*\in C({H^0})$, $S^*>0$ on $H^0$, and $\|S(\cdot,t)-S^*\|_{L^{\infty}(H^0)}+\|S(\cdot,t)-S^*\|_{L^1(H^+)}\to 0$ as $t\to\infty$.

Finally, we suppose  that $H^+\ne\emptyset$ and proceed by contradiction to show that \begin{equation}\label{Revised-eq2}
\text{meas}(\{x\in H^{+}\, |\, S^*(x)=0\})>0.\end{equation}
 So, suppose to the contrary that \eqref{Revise-eq1} does not hold. Consider the function
$$
F(x,t)=\frac{1}{t}\int_0^t\frac{I(x,s)}{S(x,s)+I(x,s)}ds,\quad \forall\ x\in \Omega, \ t>0.
$$
It is clear that 
$$ 
0\le F(x,t)\le 1,\quad \forall\ x\in\Omega,\ t>0.
$$
Moreover, for each $x\in \Omega$ satisfying $S^*(x)>0$, it holds that 
$$
\lim_{t\to\infty} F(x, t)=\lim_{t\to\infty}\frac{I(x,t)}{S(x,t)+I(x,t)}=0.
$$
Hence, since $S^*>0$ on $H^0$ and $ \text{meas}(\{x\in H^{+}\ |\ S^*(x)=0\})=0$, it follows from the Lebesgue dominated convergence theorem that 
\begin{equation}\label{Revised-eq3}
    \lim_{t\to\infty}\int_{\Omega}F(x,t)dx=0.
\end{equation}
Let $\varphi$ be the positive eigenfunction associated with $\sigma(d_I,\beta-\gamma)$ satisfying $\max_{x\in\bar{\Omega}}\varphi(x)=1$. By the second equation of \eqref{model-incidence0} and  \eqref{minmax}, we have 
\begin{align*}
    \frac{d}{dt}\int_{\Omega}\varphi Idx=&\int_{\Omega}d_I\varphi\Delta I dx+\int_{\Omega}\Big(\frac{\beta S}{S+I}-\gamma\Big)\varphi Idx\cr
    =& d_I\int_{\Omega}I(\cdot,t)\Delta\varphi dx+\int_{\Omega}\Big(\frac{\beta S}{S+I}-\gamma\Big)\varphi Idx\cr
    =& \sigma(d_I,\beta-\gamma)\int_{\Omega}\varphi I dx+\int_{\Omega}\beta\Big(\frac{S}{S+I}-1\Big)\varphi Idx\cr 
    =& \sigma(d_I,\beta-\gamma)\int_{\Omega}\varphi Idx-\int_{\Omega}\beta\frac{I}{S+I}\varphi Idx\cr 
    \ge & \sigma(d_I,\beta-\gamma)\int_{\Omega}\varphi Idx-\beta_M\varphi_{M}\|I\|_{L^\infty(\Omega)}\int_{\Omega}\frac{I}{S+I}dx\cr 
    \ge & \Big(\sigma(d_I,\beta-\gamma)-\frac{\beta_MC\varphi_M}{\varphi_{m}}\int_{\Omega}\frac{I}{S+I}dx\Big)\int_{\Omega}\varphi Idx.
\end{align*}
By the comparison principle for the ODE, we obtain that 
\begin{equation}\label{Revised-eq4}
\int_{\Omega}\varphi(x) I(x,t)dx\ge e^{t\Big(\sigma(d_I,\beta-\gamma)-M^*\int_{\Omega}F(x,t)dx\Big)}\int_{\Omega}\varphi(x) I_0(x)dx,\quad t>0,
\end{equation}
where $M^*:={\beta_MC\varphi_M}/{\varphi_{m}}$. Note that $ \sigma(d_I,\beta-\gamma)>0$, since $\beta\ge \gamma$ and $H^+\ne\emptyset$. Hence, it follows from \eqref{Revised-eq3}-\eqref{Revised-eq4} that 
$\int_{\Omega}\varphi I(\cdot,t)dx\to \infty$ as $t\to\infty$. This  contradicts the fact that $\sup_{t\ge 1}\|I(\cdot,t)\|_{L^{\infty}(\Omega)}<\infty$. Therefore, \eqref{Revised-eq2} holds.
}

{\rm (ii)} %Let $\kappa=(\beta-\gamma)/\gamma$. Define $$ V(S, I)=\frac{1}{2}\int_\Omega\left(\kappa S^2+I^2 \right)dx.  $$ It is easy to check that \begin{equation}\label{ly1st}\frac{d}{dt}V(S, I)=-d_I\int_\Omega |\triangledown I|^2dx-\int_\Omega \gamma \frac{(\kappa S-I)^2}{S+I}Idx.\end{equation}
 Integrating \eqref{ly1st} over $(0, t)$ and taking $t\to\infty$, we find that 
\begin{equation}\label{I222}
   \int_0^\infty\int_\Omega |\triangledown I|^2dxdt<\infty.  
\end{equation}
By \eqref{bound-4}, the parabolic estimates and the Sobolev embedding theorem, $I\in C^{1+\alpha, 1+\alpha/2}(\bar\Omega\times [1, \infty)])$. So by \eqref{I222} and Lemma \ref{lemma_inf0}, $\|\triangledown I(\cdot, t)\|_{L^2(\Omega)}\to 0$ as $t\to\infty$. Moreover, let  $\omega_I:=\cap_{t\ge 1} \overline{\cup_{s\ge t}  I(\cdot, s))}$, where  the completion is in $C(\bar\Omega)$. Then $\omega_I$ is well-defined, compact, and consists with constants. 

Suppose to the contrary that $0\not\in \omega_I$. By the compactness of $\omega_I$, there exists $\varepsilon_0>0$ such that  
\begin{equation}\label{Iinfe}
\liminf_{t\to\infty}\min_{x\in\bar\Omega} I(x, t)>\varepsilon_0.
\end{equation}
By the first equation of \eqref{model-incidence0}, we have
\begin{equation}\label{S-eq}
 \partial_t S=\frac{(\gamma I-(\beta-\gamma)S)I}{S+I}.
 \end{equation}
 This combined with \eqref{Iinfe} implies that $\liminf_{t\to\infty} S(x, t)\ge \varepsilon_0\gamma/(\beta-\gamma)$ pointwise for $x\in\Omega\backslash H^0$ as $t\to\infty$. By Fatou's Lemma, we have 
 $$
 \int_{\Omega\backslash H^0} \left(\frac{\varepsilon_0\gamma}{\beta-\gamma}+\varepsilon_0\right)dx\le\int_\Omega \liminf_{t\to\infty} (S+I)dx\le \liminf_{t\to\infty} \int_\Omega (S+I)dx=N.
 $$
Since $H^0$ has measure zero and $\int_\Omega 1/(\beta-\gamma)dx=\infty$, the first term in the above inequality equals infinity. This is a contradiction, and therefore $0\in \omega_I$.
Hence there exists $\{t_k\}$ converging to infinity such that $I(\cdot, t_k)\to 0$ in $C(\bar\Omega)$ as { $k\to\infty$.   Thus, $\int_{\Om}S(\cdot,t_k)dx=N-\int_{\Om}I(\cdot,t_k)dx\to N$ as $k\to\infty$.} 
\end{proof}

%\begin{remark}   Some of the results are proved in \cite{LouSalako2021}\end{remark}

\subsection{Limiting the movement of infected people}
Then, we consider the impact of limiting the movement of infected people on \eqref{model} with standard incidence mechanism by setting $d_I=0$, i.e.,
\begin{equation}\label{model-incidence1}
\begin{cases}
\displaystyle\partial_t S=d_S\Delta S-\beta(x) \frac{S I}{S+I}+\gamma(x)  I, &x\in\Omega, t>0,\\
\displaystyle\partial_t I=\beta(x) \frac{S I}{S+I}-\gamma(x)  I,&x\in\bar\Omega, t>0,\\
\partial_\nu S=0, &x\in\partial\Omega,t>0,\\
S(x, 0)=S_0(x), \ I(x, 0)=I_0(x), &x\in\bar\Omega.
\end{cases}
\end{equation}

We will establish a prior bound for the solution of \eqref{model-incidence1} first.

\begin{proposition}\label{Pr1}
Suppose that {\rm (A1)-(A2)} holds and $d_S>0$. Then \eqref{model-incidence1} has a unique nonnegative global solution $(S, I)$, where  $$S\in C([0, \infty), C(\overline{\Omega}))\cap C^1((0, \infty), {\rm Dom}_{\infty}(\Delta))\quad \text{and}\quad  I\in C^1([0, \infty), C(\overline{\Omega})).$$ Moreover, there exists $M>0$ depending on (the $L^\infty$ norm of) initial data such that 
\begin{equation}\label{bound}
 \|S(\cdot, t)\|_{L^\infty(\Omega)}, \|I(\cdot, t)\|_{L^\infty(\Omega)}\le M, \ \ \text{for all }  t\ge 0.
\end{equation}
\end{proposition}
\begin{proof}
Similar to Proposition \ref{prop-incidence},   \eqref{model-incidence1} has a unique local nonnegative  solution $(S, I)$,  where  $S\in C([0, T), C(\overline{\Omega}))\cap C^1((0, T), {\rm Dom}_{\infty}(\Delta))$ and $I\in C^1([0, T), C(\overline{\Omega}))$ for some $T>0$. It remains to show the boundedness of the solution. 

We claim that for any nonnegative integer $k$ there exists $C>0$ depending on initial data such that 
$\|S(\cdot, t)\|_{L^{2^{k}}(\Omega)}, \|I(\cdot, t)\|_{L^{2^{k}}(\Omega)}\le C$ for all $t\ge 0
$. We prove this claim by induction. It is easy to see that the claim holds for $k=0$. Now we assume that the claim holds for $k$ and will show that it holds for $k+1$. To see it, multiplying both sides of the second equation of \eqref{model-incidence1} by $I^{2^{k+1}-1}$ and integrating over $\Omega$, we obtain
\begin{eqnarray}\label{ik0}
    \frac{1}{2^{k+1}}\frac{d}{dt}\int_\Omega  I^{2^{k+1}}dx&\le& \beta_M \int_\Omega  \frac{S  I^{{2^{k+1}}-1} I}{S+I}dx-\gamma_m\int_\Omega  I^{2^{k+1}}dx \nonumber\\
    &\le& C_0\int_\Omega  S^{2^{k+1}}dx-\frac{\gamma_m}{2}\int_\Omega  I^{2^{k+1}}dx,
\end{eqnarray}  
where we have used Young's inequality in the last step. 

Multiplying both sides of the first equation of \eqref{model-incidence1} by $S^{2^{k+1}-1}$ and integrating over $\Omega$, we obtain
     \begin{eqnarray}\label{ik}
    \frac{1}{2^{k+1}}\frac{d}{dt}\int_\Omega S^{2^{k+1}}dx&=&-\frac{2^{k+1}-1}{2^{k+1}}d_S\int_\Omega |\triangledown S^{2^k}|^2dx- \int_\Omega \frac{\beta S^{2^{k+1}}I}{S+I} pdx+\int_\Omega \gamma S^{2^{k+1}-1}I dx\nonumber\\
   &\le& -C_1 \int_\Omega |\triangledown S^{2^k}|^2dx+C_2\left(\int_\Omega  S^{2^{k+1}} dx+\int_\Omega  I^{2^{k+1}} dx\right),
   \end{eqnarray}
where we have used Young's inequality in last  step. 

Multiplying \eqref{ik} by $\ds C_3:=\frac{\gamma_m}{4C_2}$ and summing up with \eqref{ik0}, we have
\begin{equation*}
\frac{d}{dt}\int_\Omega \left(C_3S^{2^{k+1}}+I^{2^{k+1}}\right)dx  \le -C_4 \int_\Omega |\triangledown S^{2^k}|^2dx+ C_5\int_\Omega  S^{2^{k+1}} dx-C_6\int_\Omega  I^{2^{k+1}} dx.
\end{equation*}
By the following interpolation inequality
  \begin{eqnarray*}
\|u\|_{L^2(\Omega)}^2\le \epsilon \|\triangledown u\|^2_{L^2(\Omega)}+C_\epsilon \|u\|_{L^1(\Omega)}^2, \ \ \forall u\in H^1(\Omega),
\end{eqnarray*}
we obtain 
\begin{equation}\label{si}
\frac{d}{dt}\int_\Omega \left(C_3S^{2^{k+1}}+I^{2^{k+1}}\right)dx  \le  C_7\left(\int_\Omega  S^{2^{k}} dx\right)^2-C_8\int_\Omega \left(C_3S^{2^{k+1}}+I^{2^{k+1}}\right) dx.
\end{equation}
By the assumption that the claim holds for $k$ and \eqref{si} there exists $C>0$ such that 
$\|S(\cdot, t)\|_{L^{2^{k+1}}(\Omega)}, \|I(\cdot,t)\|_{L^{2^{k+1}}(\Omega)}\le C$. This proves the claim. 

Fixing $p>N+2$, by the claim and the parabolic $L^p$ estimate,  there exists $C>0$ such that $\|S\|_{W^{2, 1}_p(\Omega\times (\tau, \tau+1))}<C$ for any $\tau>0$. Since $W^{2, 1}_p(\Omega\times (\tau, \tau+1))$ can be embedded into $C(\bar\Omega\times [\tau, \tau+1])$, we obtain the boundedness of $S$. We rewrite the equation of $I$ as
\begin{equation}\label{stIe}
\partial_t I=\frac{((\beta-\gamma)S-\gamma I)I}{S+I}.
\end{equation}
Hence, $\|I(\cdot, t)\|_{L^\infty(\Omega)}\le \max_{0\le s\le t}\{\|I_0\|_{L^\infty(\Omega)}, \  \|(\beta+\gamma)S(\cdot, s)\|_{L^\infty(\Omega)}/\gamma_m\}$ for any $t\ge 0$. This proves the boundedness of $I$.
\end{proof}

We are ready to study the asymptotic behavior of the solution of \eqref{model-incidence1}. Let 
$$
\kappa=\frac{\gamma}{(\beta-\gamma)}\chi_{H^+\cap\{I_0>0\}}
$$
and define 
$$
V(S, I)=\frac{1}{2} \int_\Omega\left(S^2+\kappa I^2 \right)dx. 
$$
It is easy to check that 
\begin{align}\label{vvv}
\ds\dot V(S, I)=&-d_S\int_\Omega|\triangledown S|^2dx+\int_{H^-\cup H^0\cup \{I_0=0\} } S\left(-{\beta}\frac{SI}{S+I}+\gamma I\right)dx\cr &-\int_{H^+\cap \{I_0>0\}} (\beta-\gamma)_+\frac{(S-\kappa I)^2}{S+I}Idx.
\end{align}
The function $V$ does not satisfy $\dot V\le 0$ (the second term on the right hand side of \eqref{vvv} is positive), but it still enables us to conclude the convergence of the solution.

\begin{theorem}\label{theorem_stddi}
Suppose that {\rm (A1)-(A2)} holds and $d_S>0$. Let $(S, I)$ be the solution of \eqref{model-incidence1}. Then the following statements hold:

\begin{itemize}
    \item [\rm (i)] There is a positive number $\overline{S}$ such that 
\begin{equation}\label{xx-1}
\sup_{t\ge 0}\left(\|S(\cdot, t)\|_{L^\infty(\Omega)}+\|I(\cdot, t)\|_{L^\infty(\Omega)}\right)\le \overline{S}.
\end{equation}
Furthermore, there exist  $t_1>0$ and  $\underline{S}>0$, independent of initial data, such that
\begin{equation}\label{xx-2}
   \underline{S}\le \min_{x\in\overline{\Omega}}S(x,t)\leq \max_{x\in\overline{\Omega}}S(x,t)\le \overline{S}, \quad \forall\ t\ge t_1,
\end{equation} 
and
\begin{align}\label{xx-3}
    (R-1)_+\underline{S}\chi_{\{I_0>0\}}\le & \liminf_{t\to\infty}I(x,t)\cr
    \leq & \limsup_{t\to\infty}I(x,t)\leq (R-1)_+\overline{S}\chi_{\{I_0>0\}}, \quad \forall\ x\in\bar\Omega, 
\end{align}
{  where $R:=\beta/\gamma$}.

\item[\rm (ii)] If $1/(\beta-\gamma)\in L^1(H^+\cap\{I_0>0\})$, then 
$$
\lim_{t\to\infty} (S(x, t), I(x, t)=(S^*, I^*)), \ \ \text{uniformly for}\  x\in\bar\Omega,
$$
 where 
 $$
 I^*=\frac{(\beta-\gamma)_+S^*}{\gamma}\chi_{H^+\cap\{I_0>0\}}
 $$
 and $S^*$ is a positive constant given by 
\begin{equation}\label{xx-10}
S^*=\frac{N}{|\Omega|+\int_{H^+\cap\{I_0>0\}} \frac{(\beta-\gamma)_+}{\gamma} dx}.
\end{equation}
\end{itemize}
\end{theorem}
\begin{proof}

{\rm (i)} Note that \eqref{xx-1} follows from Proposition \ref{Pr1}. Next, observe that
\begin{align*}
\frac{d}{dt}\int_{\Omega}Sdx=&\int_{\Omega}\gamma Idx-\int_{\Omega}\beta\frac{I}{I+S}Sdx\cr 
\ge& \gamma_m\int_{\Omega}Idx-\beta_M\int_{\Omega}Sdx\cr = &\gamma_m
\left(N-\int_{\Omega}Sdx\right)-\beta_M\int_{\Omega}Sdx\cr
=&\gamma_mN-(\beta_M+\gamma_m)\int_{\Omega}Sdx.
\end{align*}
Therefore, by $\int_{\Omega}S_0dx\ge 0$ and  the comparison principle, we have 
that 
\begin{equation}\label{xx-6}
\int_{\Omega}S(x, t)dx\ge \frac{\gamma_mN}{\beta_M+\gamma_m}(1-e^{-t(\beta_M+\gamma_m)}),\quad \forall\ t\ge 0.
\end{equation}
Next by \eqref{model-incidence1}, we see that 
$$
\begin{cases}
    \partial_tS\ge d_S\Delta S-\beta_MS, & x\in\Omega, t>0, \cr 
    \partial_{\nu}S=0, & x\in\partial\Omega,\ t>0.
\end{cases}
$$
For any $t_0>0$, it  follows from the comparison principle for parabolic equations that 
\begin{equation}\label{LKO1}
S(\cdot, t+t_0)\ge e^{-t\beta_M}e^{td_S\Delta}S(\cdot, t_0),\quad \forall\ t>0.
\end{equation}
Thanks to the Harnack's inequality (see Lemma \ref{lemma_harnack}), there is a positive constant $c_0$ such that 
\begin{equation}\label{xx-7}
e^{td_S\Delta}S(\cdot, t_0)(x)\ge c_0e^{td_S\Delta}S(\cdot, t_0)(y),\quad \forall\ td_S\ge 1, \ x,y\in\bar\Omega,\ t_0>0.
\end{equation}
This in turn together with \eqref{xx-6} implies that
\begin{align*}
\min_{x\in\overline{\Omega}}e^{d_St\Delta}S(\cdot,t_0)(x)\ge & \frac{c_0}{|\Omega|}\int_{\Omega}e^{td_S\Delta}S(\cdot,t_0)(y)dy\cr 
= & \frac{c_0}{|\Omega|}\int_{\Omega}S(y,t_0)dy\cr 
\ge &  \frac{\gamma_m Nc_0}{|\Omega|(\beta_M+\gamma_m)}(1-e^{-t_0(\beta_M+\gamma_m)}),\quad \quad t\ge \frac{1}{d_S},\ t_0>0.
\end{align*}
As a result, it follows from \eqref{LKO1} that 
{
$$
S\left(x,\frac{1}{d_S}+t_0\right)\ge \frac{\gamma_{m}Nc_0 e^{-\frac{\beta_M}{d_S}}}{|\Omega|(\beta_M+\gamma_{m})}(1-e^{-t_0(\beta_M+\gamma_m)}),\quad \quad  \ x\in\bar\Omega, t_0>0.
$$
}
Therefore taking $t_1=\frac{1}{d_S}+\frac{\ln(2)}{\beta_M+\gamma_{m}}$,   \eqref{xx-2} holds with $\underline{S}:=\frac{\gamma_mNc_0 e^{-\frac{\beta_M}{d_S}}}{2|\Omega|(\beta_M+N\gamma_{m})}>0$, where $\underline{S}$ and $t_1$ are independent of the initial data.

Next, we show that \eqref{xx-3} holds. To this end, we first note that 
\begin{align*}
    I_t= \beta\Big((1-r)-\frac{I}{I+S}\Big)I\le & \beta\Big((1-r)-\frac{I}{I+\overline{S}}\Big)I\cr 
    =& \gamma\Big((R-1)\overline{S}-I\Big)\frac{I}{I+\overline{S}}, \ \quad t>0,
\end{align*}
which in view of the comparison principle for ordinary differential equations implies that $\limsup_{t\to\infty}I(x,t)\le (R-1)_{+}\overline{S}$. Recalling that $I(x,t)=0$ for all $t>0$ whenever $I_0(x)=0$, we then conclude that $\limsup_{t\to\infty}I(x,t)\le (R-1)_{+}\overline{S}\chi_{\{I_0>0\}}$.

Similarly, observing that 
\begin{align*}
    I_t= \beta\Big((1-\frac{\gamma}{\beta})-\frac{I}{I+S}\Big)I\ge & \beta\Big((1-r)-\frac{I}{I+\underline{S}}\Big)I\cr 
    =& \gamma\Big((R-1)\underline{S}-I\Big)\frac{I}{I+\underline{S}},\quad t>t_1,
\end{align*}
we can proceed as in the previous case to establish that $\liminf_{t\to\infty}I(x,t)\ge (R-1)_{+}\underline{S}\chi_{\{I_0>0\}}$, which completes the proof of {\rm (i)}.

{\rm (ii)} By the equation of $I$, we have 
\begin{equation*}
\int_0^t\int_{H^-\cup H^0\cup\{I_0=0\}} \left(\beta\frac{SI}{S+I}-\gamma I\right) dxds=\int_{H^-\cup H^0\cup \{I_0= 0\}} I(x, t)dx-\int_{H^-\cup H^0\cup\{I_0= 0\}} I_0dx.
\end{equation*}
Note that 
$$
\beta \frac{SI}{S+I}-\gamma I=\frac{((\beta-\gamma)S-{ \gamma}I)I}{S+I}\le 0, \ \ \forall x\in H^-\cup H^0\cup\{I_0=0\}.
$$
We have 
$$
0\le \int_0^\infty\int_{H^-\cup H^0\cup \{I_0=0\}} \left(-\beta\frac{SI}{S+I}+\gamma I\right) dxds<\infty,
$$
which together with \eqref{xx-1} yields that 
$$
0\le \int_0^\infty\int_{H^-\cup H^0\cup\{I_0=0\}} S\left(-{\beta}\frac{SI}{S+I}+\gamma I\right)dxdt<\infty.
$$
Hence integrating \eqref{vvv} over $(0, \infty)$, we obtain that 
\begin{equation}\label{cv-1}
    \int_0^\infty \int_\Omega|\triangledown S|^2dxdt<\infty
\end{equation}
and
\begin{equation}\label{cv-2}
    \int_0^\infty \int_{H^+\cap\{I_0>0\}} (\beta-\gamma)_+I\frac{(S-\kappa I)^2}{S+I}dxdt<\infty.
\end{equation}

By \eqref{xx-1}, we have 
$$
\sup_{t\ge 0}\|\frac{\beta SI}{S+I}-\gamma I\|_{L^\infty(\Omega)}<\infty.
$$ 
So
by the regularity theory for parabolic equations and Sobolev embedding theorem, the mappings $\nabla S(x,t) $ and $S$ are H\"older continuous on $\overline{\Omega}\times [1,\infty)$, and $\{S(\cdot,t)\}_{t\ge 1}$ is precompact in $C^{1+\alpha}(\Omega)$, $0<\alpha<1$. Then by \eqref{cv-1} and Lemma \ref{lemma_inf0}, $\int_{\Omega}|\nabla S|^2dx\to 0$ as $t\to\infty$. Hence the set $w_S:=\cap_{t\ge 1}\overline{\cup_{s\ge t}\{S(\cdot,s)\}}$ consists of positive constant functions. Furthermore since $ \sup_{t\ge 0}\|\partial_t I(\cdot,t)\|_{L^\infty(\Omega)}<\infty$ and $S$ is H\"older continuous on $\overline{\Omega}\times [1,\infty)$,   the mapping
$$
t\mapsto \int_{H^+\cap\{I_0>0\}} (\beta-\gamma)_+\frac{(S-\kappa I)^2}{S+I}Idx
$$ 
is H\"older continuous on $[1,\infty)$. Therefore by Lemma \ref{lemma_inf0} and \eqref{cv-2},  we have
\begin{equation}\label{xx-8}
    % \lim_{t\to\infty}\int_{\Omega}|\nabla S|^2=0\quad \text{and}\quad 
  \lim_{t\to\infty}\int_{H^+\cap\{I_0>0\}} (\beta-\gamma)_+\frac{(S-\kappa I)^2}{S+I}Idx=0.
\end{equation}
% Hence, in view of \eqref{xx-2}, the set $w_S:=\cap_{t\ge 1}\overline{\cup_{s\ge t}\{S(\cdot,s)\}}$ consists of positive constant functions. 
Now, let $S^*\in w_S$. Then there is a sequence $t_k\to \infty$ such that $S(\cdot,t_k)\to S^*$ uniformly on $\bar\Omega$ as {$k\to\infty$}. Since \eqref{xx-8} holds,   after passing to a subsequence if necessary,  we may suppose that 
\begin{equation}\label{xx-9}
    \lim_{k\to\infty}I(x,t_k)\frac{(S^*-\kappa(x)I(x,t_k))^2}{S^*+I(x,t_k)}=0 \quad \text{ a.e. on }\ H^+\cap\{I_0>0\}.
\end{equation}
However, when $x\in H^+\cap \{I_0>0\}$, we have from \eqref{xx-3} that $\liminf_{t\to\infty}I(x,t)\ge \underline{S}(\beta(x)-\gamma(x))>0$. Therefore, we conclude from \eqref{xx-9} that 
$$
\lim_{{ k}\to\infty}I(x,t_k)=\frac{S^*}{\kappa(x)}=\frac{(\beta(x)-\gamma(x))S^*}{\gamma(x)} \quad \text{ a.e. on }\ H^+\cap \{I_0>0\}.
$$
By \eqref{xx-3} again, $\lim_{t\to\infty}I(x,t)=0$ almost everywhere for  $x\in\Omega\setminus(H^+\cap \{I_0>0\})$. So by the dominated convergence theorem, we have
\begin{eqnarray*}
N&=&\lim_{k\to\infty}\int_{\Omega}(S(\cdot,t_k)+I(\cdot,t_k))dx\\
&=&|\Omega|S^*+S^*\int_{H^+\cap \{I_0>0\}}\frac{1}{\kappa}dx=\Big(|\Omega|+\int_{H^+\cap\{I_0>0\}}\frac{\beta-\gamma}{\gamma}dx\Big)S^*.
\end{eqnarray*}
This yields that 
$$
S^*=\frac{N}{|\Omega|+\int_{H^+\cap\{I_0>0\}}\frac{\beta-\gamma}{\gamma}dx}.
$$
Since $S^*$ is independent of the chosen subsequence, we have
$$
w_S=\Big\{\frac{N}{|\Omega|+\int_{H^+\cap\{I_0>0\}}\frac{\beta-\gamma}{\gamma}dx} \Big\},
$$ 
and therefore \eqref{xx-10} holds.  Since $S(\cdot,t)\to S^*$  uniformly on $\Omega$ as $t\to\infty$, we obtain from \eqref{stIe} that 
$I(\cdot,t)\to \frac{S^*(\beta-\gamma)^+}{\gamma}\chi_{H^+\cap\{I_0>0\}}$ uniformly on $\bar\Omega$ as $t\to\infty$.
\end{proof}

\section{Simulations}
In this section, we run numerical simulations to illustrate the results. Let $\Omega=[0, 1]$, $S_0=2+\cos(\pi x)$ and $I_0=1.5+\cos(\pi x)$. Then the total population is $N=\int_\Omega (S_0+I_0)dx=3.5$. %In all the  simulations below, we plot the solutions $(S, I)$ at time $t=10$. 

\subsection{Mass action mechanism}
We first simulate the models with  mass action mechanism. 

\noindent  \textbf{Simulation 1: control the movement of susceptible people}. Let $d_S=0$, $d_I=1$, and $\gamma=4-\pi \sin(\pi x)$. Firstly, choose $\beta=0.5$, and so $N<\int_\Omega \gamma/\beta dx=4$, i.e. the total population is small. By Theorem \ref{theorem_massds}, we have $I(\cdot, t)\to 0$ as $t\to\infty$ and the infected population will be eliminated, which is confirmed by Figure \ref{fig1a}. Then, choose $\beta=2$, and so $N>\int_\Omega \gamma/\beta dx=1$, i.e. the total population is large. Now Corollary \ref{cor1} predicts that $S(\cdot, t)\to \gamma/\beta$ and $I\to I^*={(N-\int_{\Omega}r dx)}/{|\Omega|}=2.5$, which is confirmed by Figure \ref{fig1b}. 
{  Finally, choose $\gamma=0.5(1+x)$ such that $\int_\Omega rdx=\int_\Omega \gamma/\beta dx\approx 2.82$ and 
$S_0-r$ changes sign on $\Omega$ (By Theorem \ref{theorem_massds} and Corollary  \ref{cor1}, we have already known that $N=\int_\Omega rdx$ is a threshold value for the two alternatives in Theorem  \ref{theorem_massds} if $\beta$ is a constant or $S_0-r$ does not change sign on $\Omega$). Replace the initial condition by $I_0=a+cos(\pi x)$ and then $N=\int_\Omega (S_0+I_0)dx=a+2$. Figure \ref{fig1bN} shows $\int_\Omega I(x, 40)dx$ as a function of $N$, which indicates that a bifurcation appears at $N\approx 2.82$, i.e.  $N\approx\int_\Omega r dx$ is a threshold value for alternative (i) vs (ii) in Theorem \ref{theorem_massds}. It is still an open problem to rigorously show whether $N=\int_\Omega rdx$ is the threshold value for the two alternatives in Theorem \ref{theorem_massds} without the additional assumptions in Corollary  \ref{cor1}.} From the simulations, we can see that the disease can be eliminated only when the total population is small by controlling the movement of susceptible people.

\begin{figure}
     \centering
     \begin{subfigure}[b]{0.3\textwidth}
         \centering
          \includegraphics[scale=.3]{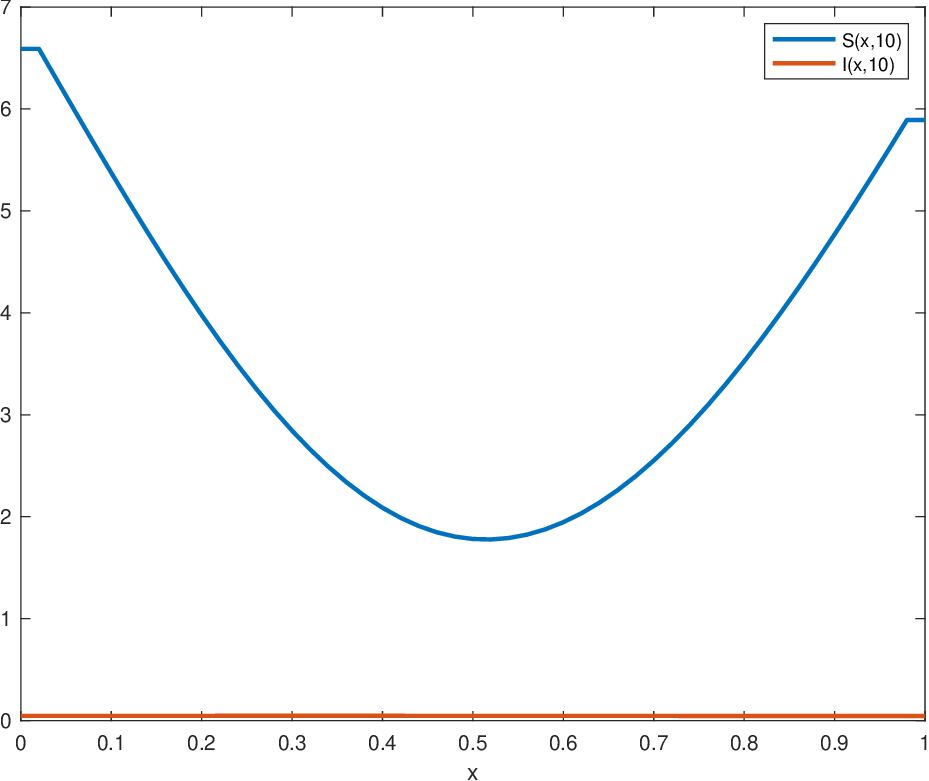}
         \caption{}
         \label{fig1a}
     \end{subfigure}
     \begin{subfigure}[b]{0.3\textwidth}
        \centering
     \includegraphics[scale=.3]{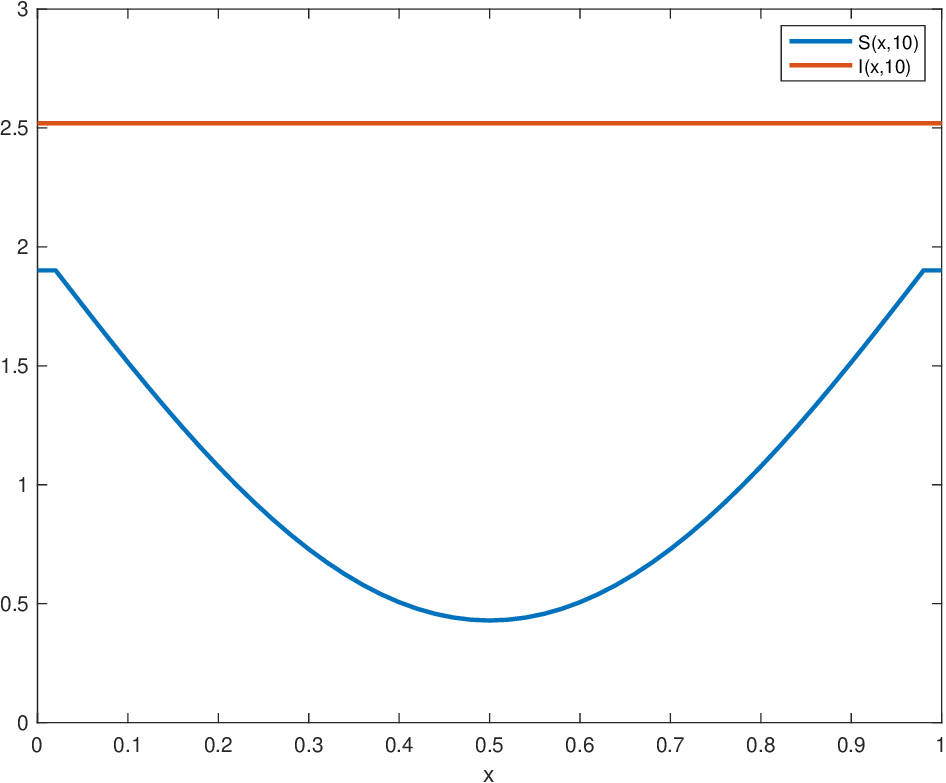}
             \caption{}
         \label{fig1b}
     \end{subfigure}
       \begin{subfigure}[b]{0.3\textwidth}
        \centering
     \includegraphics[scale=.3]{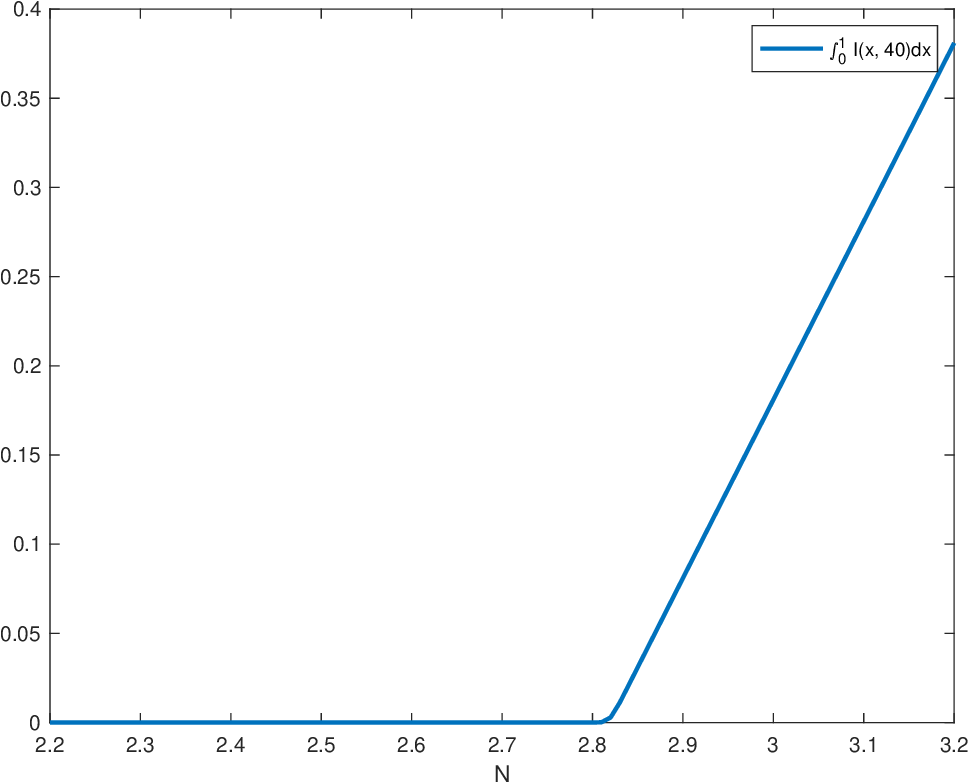}
             \caption{}
         \label{fig1bN}
     \end{subfigure}
     \caption{Simulations of the model with mass action mechanism and $d_S=0$. Parameters: $d_I=1$, $\gamma=4-\pi \sin(\pi x)$. Left figure: $\beta=0.5$ and $N<\int_\Omega \gamma/\beta dx$; middle figure: $\beta=2$ and $N>\int_\Omega \gamma/\beta dx$; {  right figure: $\beta=0.5(1+x)$ and $I_0$ is replaced by $a+cos(\pi x)$ with $a\in [0.2, 1.2]$.} }
     \label{fig1}
 %    \vskip -20pt
\end{figure}

\noindent  \textbf{Simulation 2: control the movement of infected people}. Let $d_S=1$ and $d_I=0$.  Firstly, choose $\beta=0.2$ and $\gamma=4-\pi \sin(\pi x)$ such that $H^+=\emptyset$. By Theorem \ref{theorem_S}, $S(\cdot, t)$ converges to $N/|\Omega|=3.5$ and $I(\cdot, t)$ converges to 0, which is confirmed by Figure \ref{fig2a0}. 
Then, choose $\beta=1$ and $\gamma=4-\pi \sin(\pi x)$ such that $H^+\neq \emptyset$ and  the minimum of $\gamma/\beta$ is attached at $x=0.5$. By Theorem \ref{theorem_S}, the infected people will concentrate at $x=0.5$, which is confirmed by Figure \ref{fig2a}. Finally we choose $\beta=2$ and $\gamma=14-4\pi\sin(4\pi x)$ such that  $H^+\neq \emptyset$. The minimum of $\gamma/\beta$ is attached at $x=1/8, 5/8$, and the infected people  concentrate at these two points as shown in Figure \ref{fig2b}. From the simulations, we can see that the infected people may not be eliminated by controlling the movement of infected people if  $H^+\neq \emptyset$. Instead, the infected people will concentrate at certain points that are of the highest risk.

\begin{figure}
     \centering

        \begin{subfigure}[b]{0.3\textwidth}
         \centering
          \includegraphics[scale=.3]{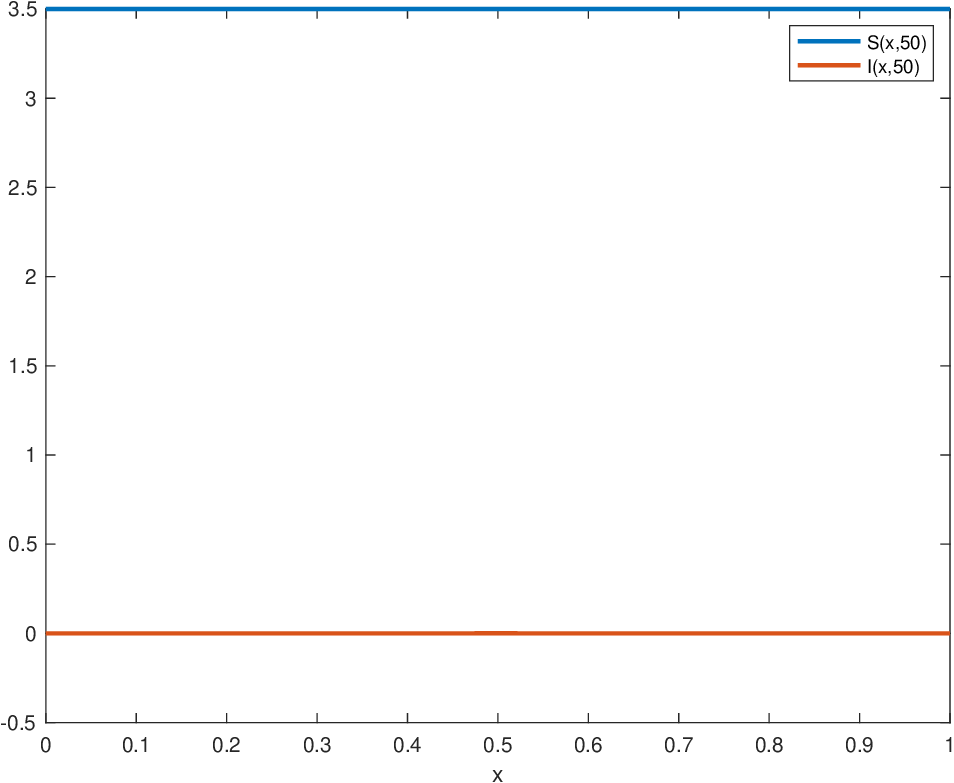}
         \caption{}
         \label{fig2a0}
     \end{subfigure}
     \begin{subfigure}[b]{0.3\textwidth}
         \centering
          \includegraphics[scale=.3]{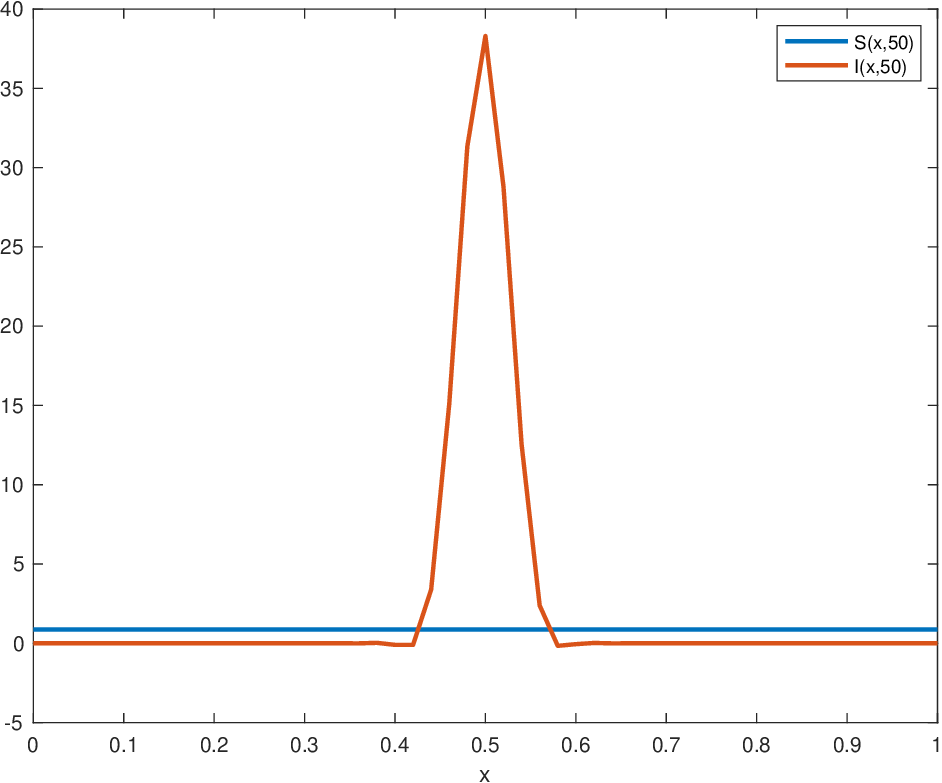}
         \caption{}
         \label{fig2a}
     \end{subfigure}  
     \begin{subfigure}[b]{0.3\textwidth}
        \centering
     \includegraphics[scale=.3]{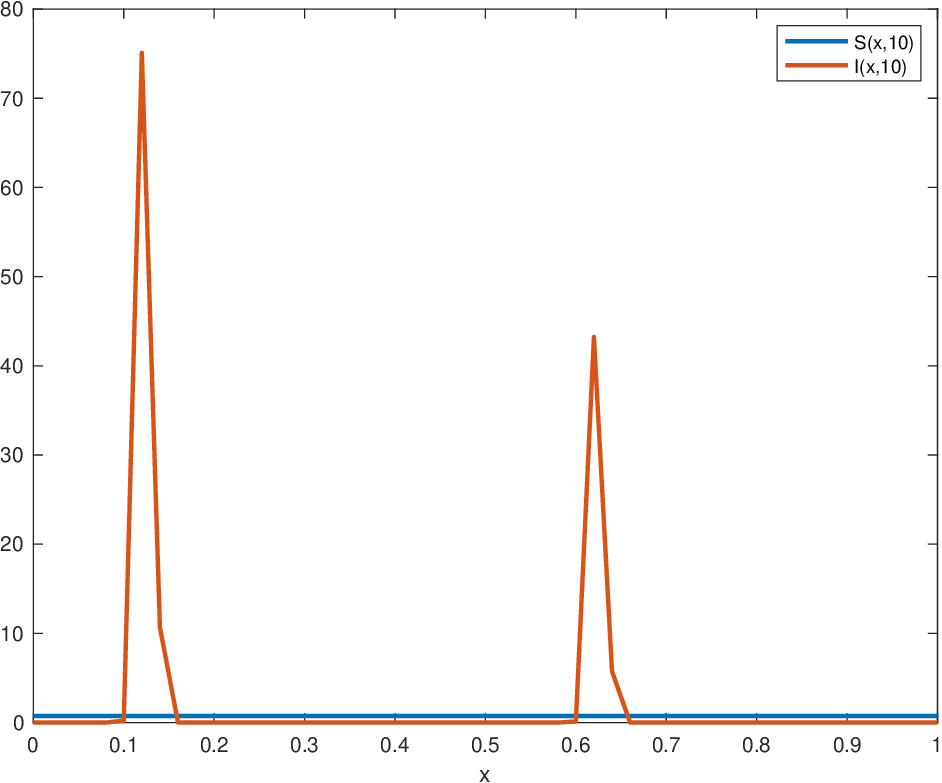}
             \caption{}
         \label{fig2b}
     \end{subfigure}
  
     \caption{Simulations of the model with mass action mechanism and $d_S=1, d_I=0$. Left figure: $\beta=0.2$, $\gamma=4-\pi \sin(\pi x)$, and $H^+=\emptyset$; middle figure: $\beta=1$, $\gamma=4-\pi \sin(\pi x)$, $H^+\neq\emptyset$, and  the minimum of $\gamma/\beta$ is attached at $x=0.5$; right figure: $\beta=2$, $\gamma=14-4\pi\sin(4\pi x)$, $H^+\neq\emptyset$, and the minimum of $\gamma/\beta$ is attached at $x=1/8, 5/8$.}
     \label{fig2}
 %    \vskip -20pt
\end{figure}

\subsection{Standard incidence mechanism}
We then simulate the models with standard incidence mechanism.

\noindent  \textbf{Simulation 3: control the movement of susceptible people}.  Let $d_S=0$ and $d_I=1$. Firstly, choose  $\beta=1+\sin(\pi x)$ and $\gamma=1.5$ such that $ \int_{\Omega}(\beta-\gamma)dx={2}/{\pi}-0.5>0$, the high-risk sites are $H^+=(1/6, 5/6)$, the low-risk sites are $H^-=[0, 1/6)\cup (5/6, 1]$ and moderate-risk sites are $H^0=\{1/6, 5/6\}$. Since $H^-\neq\emptyset$, Theorem \ref{theorem_SL}-{\rm (i)} predicts that  infected population will be eliminated  and susceptible people occupy $H^-\cup H^0$. {Moreover since $\int_{\Omega}(\beta-\gamma)>0$, then $\mathcal{R}_0>1$ and Remark \ref{Remark1} suggests that the local size of the susceptible population maybe significantly low on some portion of the high-risk area,} which is confirmed by Figure \ref{fig30a}. Then choose $\beta=2.5+\sin(\pi x)$ and $\gamma=1.5+\sin(\pi x)$  such that  $\beta>\gamma$ and $I^*=N/\int_\Omega (\beta/(\beta-\gamma))dx\approx 1.1159$.  Theorem \ref{theorem_SL}-{\rm (ii)} predicts that $I(\cdot, t)\to I^*$ as $t\to\infty$, which is confirmed by Figure \ref{fig30b}.  Finally, choose $\beta=2-\sin(\pi x)$ and $\gamma=1$ such that $\beta\ge\gamma$, $H^0=\{0.5\}$, and $\int_\Omega 1/(\beta-\gamma)dx=\infty$. As shown in Figure \ref{fig31},  infected people are eliminated which agrees with Theorem \ref{theorem_stdDS}. Moreover,  susceptible people seem to concentrate near $H^0$, which we cannot prove.   Our theoretical results and simulations show that the disease may be controlled by limiting the movement of susceptible people if there exist low-risk or moderate-risk sites.

% Then, choose $\beta=2+|x-0.5|^{0.5}$ and $\gamma=2$ such that $\beta\ge\gamma$, $H^0=\{0.5\}$, and $\int_\Omega 1/(\beta-\gamma)dx<\infty$. The simulations in Figure \ref{fig32} cannot confirm whether the infected people will be eliminated or the susceptible people concentrate at $H^0$.  {  Salako: Can we get one simulation for the case of $H^+=\overline{\Omega}$ and one simulation for $H^{-}\ne \emptyset$ ? }

\begin{figure}
     \centering
     \begin{subfigure}[b]{0.45\textwidth}
         \centering
          \includegraphics[scale=.35]{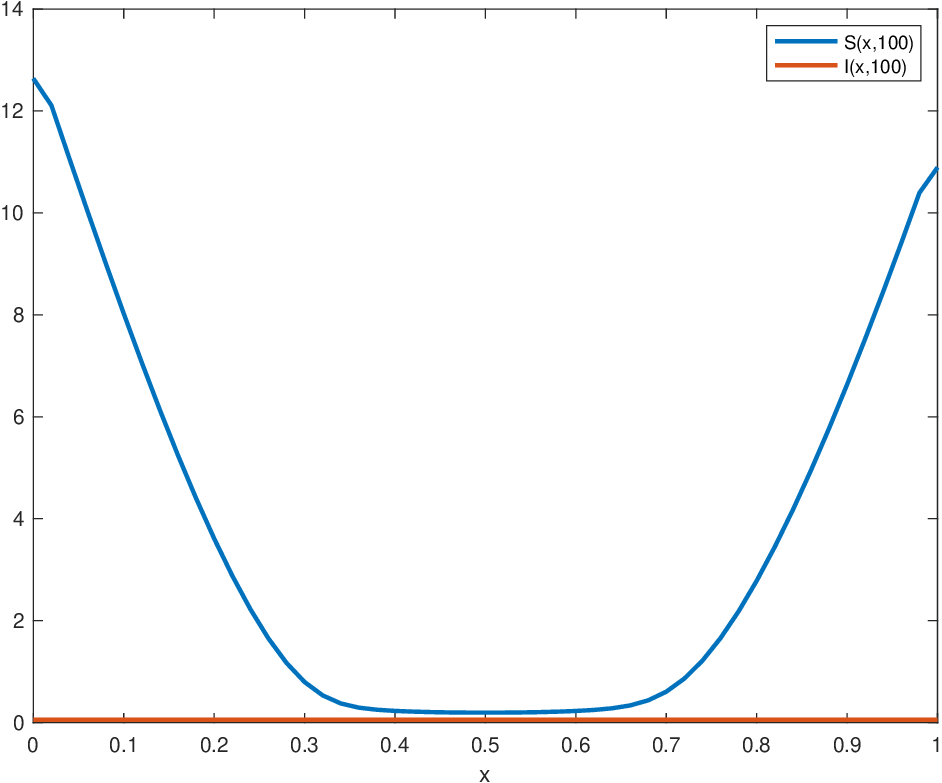}
         \caption{}
         \label{fig30a}
     \end{subfigure}
     \begin{subfigure}[b]{0.45\textwidth}
        \centering
     \includegraphics[scale=.35]{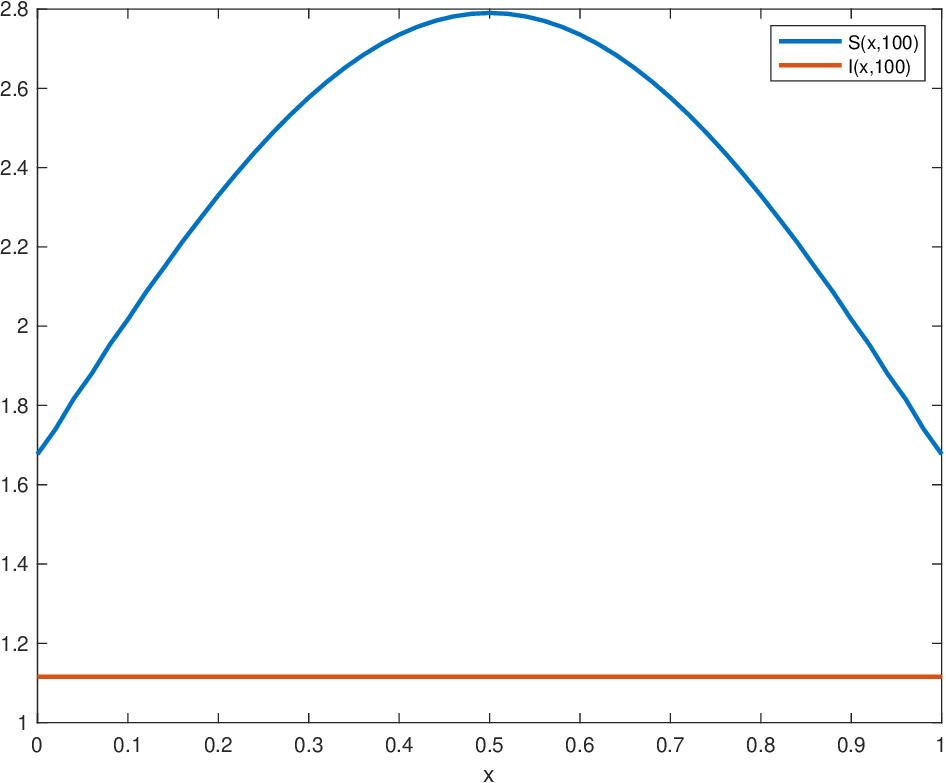}
             \caption{}
         \label{fig30b}
     \end{subfigure}
  
     \caption{Simulations of the model with standard incidence mechanism and $d_S=0, d_I=1$. Left figure: $\beta=1+\sin(\pi x)$ and $\gamma=1.5$ such that   $H^+=(1/6, 5/6)$,  $H^-=[0, 1/6)\cup (5/6, 1]$ and  $H^0=\{1/6, 5/6\}$; right figure: $\beta=2.5+\sin(\pi x)$ and $\gamma=1.5+\sin(\pi x)$  such that  $\beta>\gamma$ and $I^*=N/\int_\Omega (\beta/(\beta-\gamma))dx\approx 1.1159$.}
     \label{fig30}
 %    \vskip -20pt
\end{figure}

\begin{figure}
     \centering
     \begin{subfigure}[b]{0.3\textwidth}
         \centering
          \includegraphics[scale=.3]{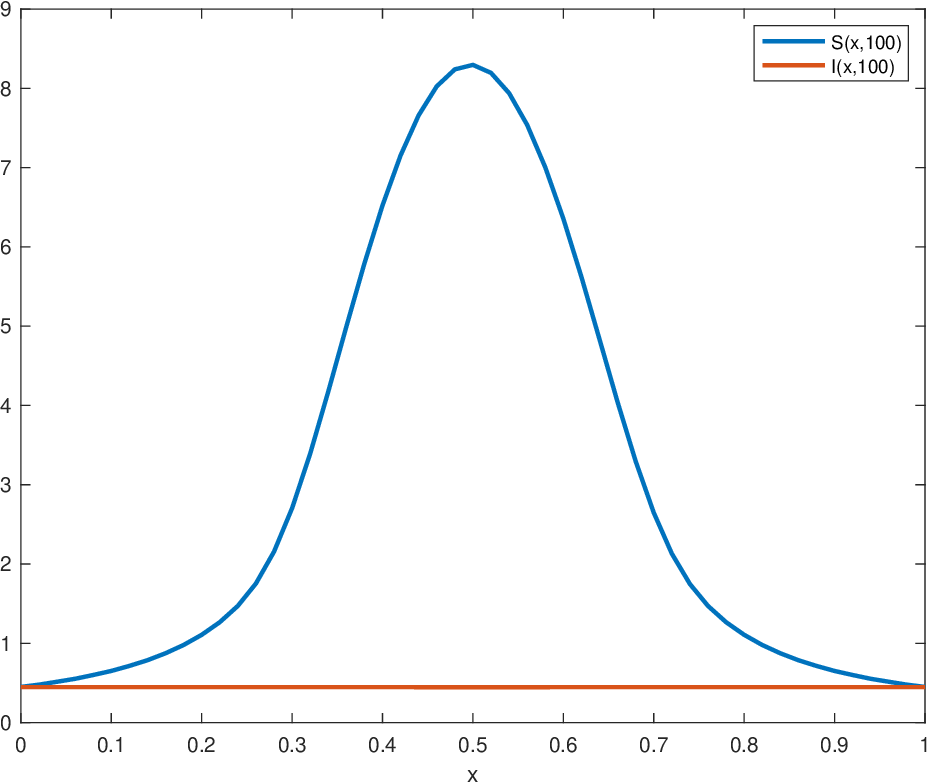}
         \caption{}
         \label{fig31a}
     \end{subfigure}
     \begin{subfigure}[b]{0.3\textwidth}
        \centering
     \includegraphics[scale=.3]{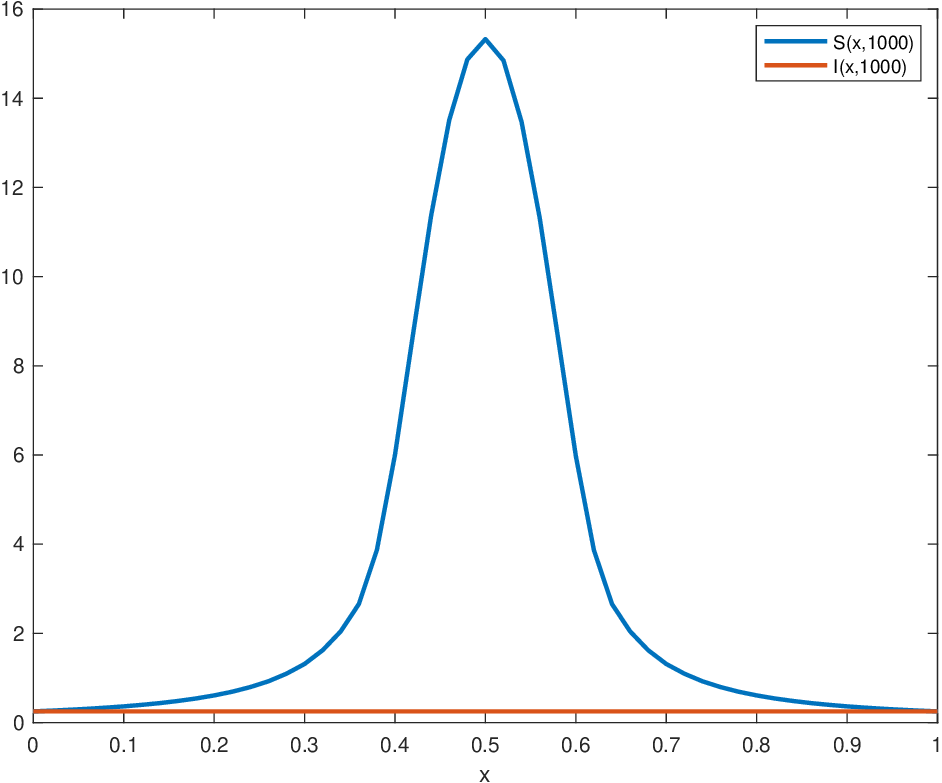}
             \caption{}
         \label{fig31b}
     \end{subfigure}
       \begin{subfigure}[b]{0.3\textwidth}
        \centering
     \includegraphics[scale=.3]{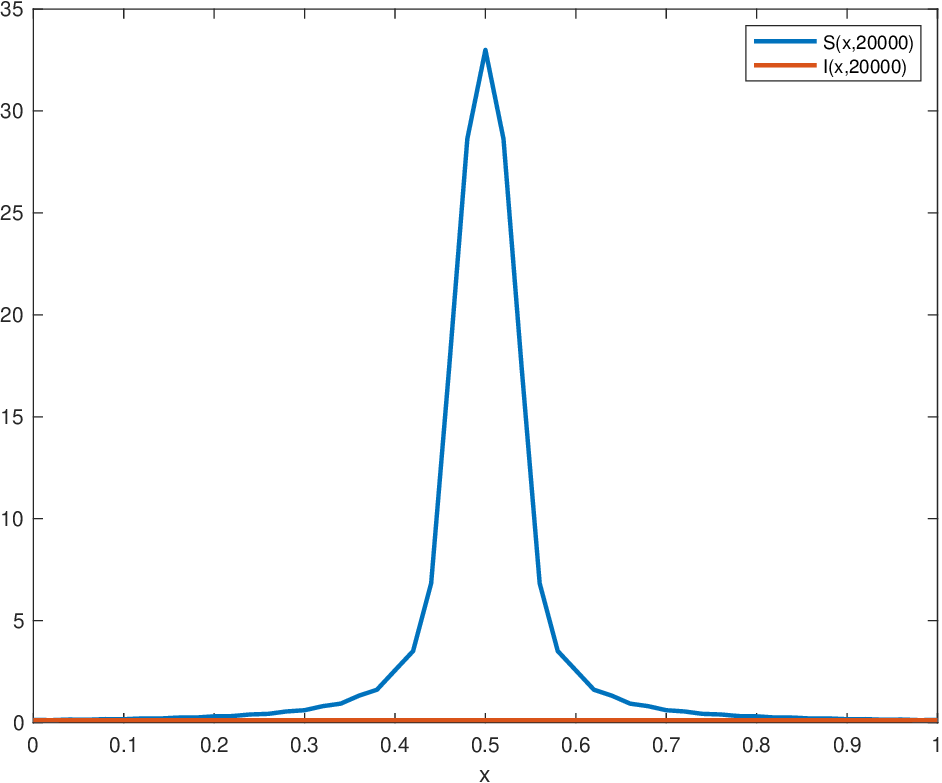}
             \caption{}
         \label{fig31c}
     \end{subfigure}
  
     \caption{Simulations of the model with standard incidence mechanism and $d_S=0, d_I=1$. Parameters: $\beta=2-\sin(\pi x)$ and $\gamma=1$ such that $\int_\Omega 1/(\beta-\gamma)dx=\infty$.}
     \label{fig31}
 %    \vskip -20pt
\end{figure}

% \begin{figure}
%      \centering
%      \begin{subfigure}[b]{0.3\textwidth}
%          \centering
%           \includegraphics[scale=.3]{stDS2T100.eps}
%          \caption{}
%          \label{fig32a}
%      \end{subfigure}
%      \begin{subfigure}[b]{0.3\textwidth}
%         \centering
%      \includegraphics[scale=.3]{stDS2T1000.eps}
%              \caption{}
%          \label{fig32b}
%      \end{subfigure}
%        \begin{subfigure}[b]{0.3\textwidth}
%         \centering
%      \includegraphics[scale=.3]{stDS2Large.eps}
%              \caption{}
%          \label{fig32c}
%      \end{subfigure}
  
%      \caption{Simulations of the model with standard incidence mechanism and $d_S=0, d_I=1$. Parameters: $\beta=2+|x-0.5|^{0.5}$ and $\gamma=2$ such that $\int_\Omega 1/(\beta-\gamma)dx<\infty$.}
%      \label{fig32}
%  %    \vskip -20pt
% \end{figure}

\noindent  \textbf{Simulation 4: control the movement of infected people}.  Let $d_S=1$ and $d_I=0$. Firstly, choose $\beta=2-|x-0.5|^{0.5}$ and $\gamma=1.5$ such that  the high-risk sites are $H^+=(0.25, 0.75)$ and $\int_\Omega 1/|\beta-\gamma| dx<\infty$. As shown in Figure \ref{fig4a}, $S(\cdot, t)$ converges to a positive constant and $I(x, t)$ is positive when $x\in H^+$.  Then we choose $\beta=2-\sin(\pi x)$ and $\gamma=1.5$  such that  the high-risk sites are $H^+=(0, 1/6)\cup (5/6, 1)$ and $\int_\Omega 1/|\beta-\gamma| dx=\infty$. The infected people will live in $H^+$ as shown in Figure \ref{fig4b}. From the simulations, we can see that the infected people may be eliminated exactly at the low-risk sites by controlling the movement of infected people, which is predicted by Theorem \ref{theorem_stddi}. 

\begin{figure}
     \centering
     \begin{subfigure}[b]{0.45\textwidth}
         \centering         \includegraphics[scale=.35]{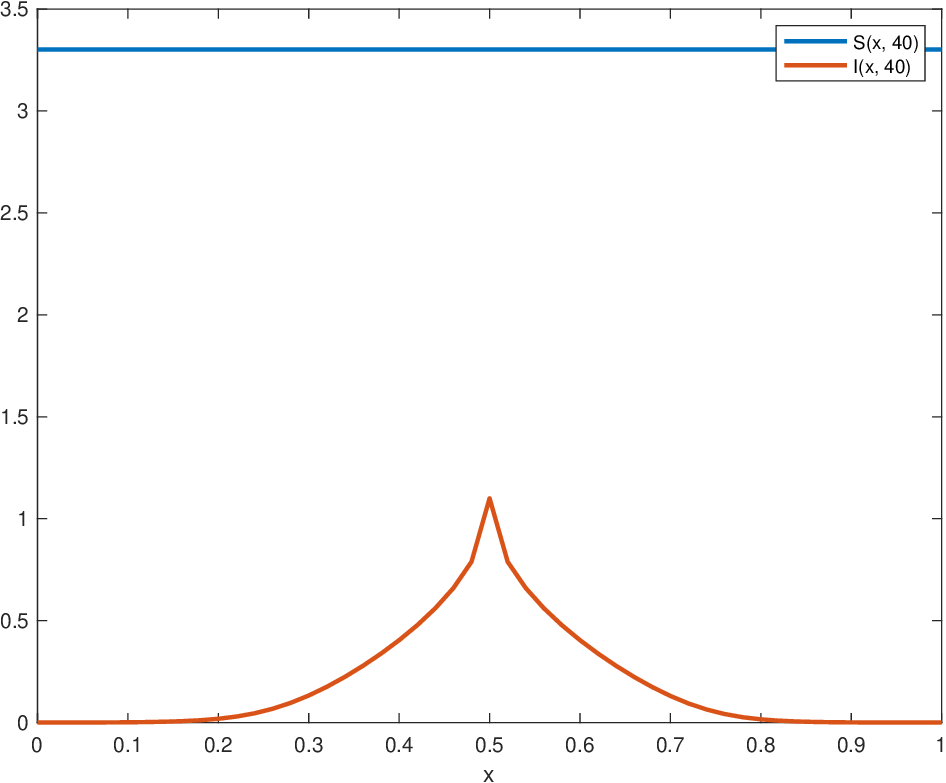}
         \caption{}
         \label{fig4a}
     \end{subfigure}
     \begin{subfigure}[b]{0.45\textwidth}
        \centering
     \includegraphics[scale=.35]{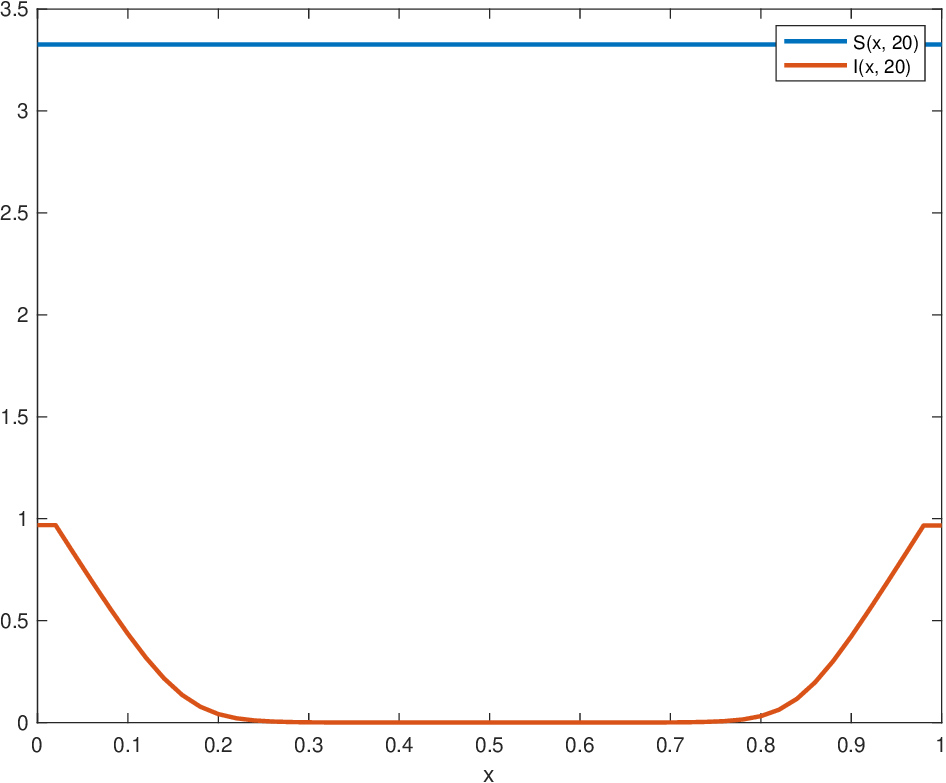}
             \caption{}
         \label{fig4b}
     \end{subfigure}
  
     \caption{Simulations of the model with standard incidence mechanism and $d_S=1, d_I=0$. Left figure: $\beta=2-|x-0.5|^{0.5}$ and $\gamma=1.5$ such that  the high-risk sites are $H^+=(0.25, 0.75)$ and $\int_\Omega 1/|\beta-\gamma| dx<\infty$.; right figure: $\beta=2-\sin(\pi x)$ and $\gamma=1.5$  such that  the high-risk sites are $H^+=(0, 1/6)\cup (5/6, 1)$ and $\int_\Omega 1/|\beta-\gamma| dx=\infty$.}
     \label{fig4}
 %    \vskip -20pt
\end{figure}

\section{Conclusions}
We studied the impact of limiting population movement on disease outbreak by examining the large time behavior of classical solutions of a class of epidemic models with  mass action  or  standard incidence transmission  mechanism. To this end,  we  set the  diffusion rate of the population subgroups (susceptible or infected)  to zero separately and presented  detailed mathematical analysis of the  corresponding degenerate epidemic models.  First, we established the existence and uniqueness of global classical solutions (see Propositions \ref{prop}, \ref{prop_ex1}, \ref{prop-incidence} and \ref{Pr1}). Next, we discussed the global dynamics of the solutions (Theorems \ref{theorem_massds}, \ref{theorem_S}, \ref{theorem_stdDS}, and \ref{theorem_stddi}). Finally, we conducted some numerical simulations to complement and illustrate our theoretical results (see Figures \ref{fig1}-\ref{fig4}).    Our results revealed the intricate effects of restricting population movement on disease dynamics and how  predictions may { depend} on the choices of transmission mechanisms and spatially heterogeneous parameters.

 First,  we considered  the global dynamics of the model with  mass action transmission mechanism. We defined a risk function $r:={\gamma}/{\beta}$ as the ratio of the recovery and  transmission rates.   When the susceptible population movement is restricted, Theorem \ref{theorem_massds} indicates that the average  population size, ${N}/{|\Omega|}$, compared to that of the  risk function, $\int_{\Omega}rdx/|\Omega|$,  largely determines the disease outcome. In particular,  regardless of the magnitude of the movement rate of the infected population, the disease may be eradicated only  if the average population size is smaller than the average  risk function. Moreover if the disease persists,  the infected population would  eventually be uniformly distributed across the whole habitat.  On the other hand, when the movement of the infected population is restricted,  Theorem \ref{theorem_S} suggests that the disease could be eradicated  only if the habitat does not have a high-risk area. Moreover, whether the magnitude of the movement rate of the susceptible population affects the results depends on whether the average size of the population is larger than  the risk function in at least  one location.   Furthermore when the disease persists, the infected population would concentrate on the highest-risk area of the habitat. Our simulations in Figures \ref{fig1}-\ref{fig2} illustrated the above theoretical results.

Next,  we studied  the global dynamics of the model with  standard incidence mechanism. Theorems \ref{theorem_SL} and \ref{theorem_stdDS}  indicate that restricting the susceptible population's movement could completely eradicate the disease  only if  the habitat  accommodates either a low-risk  or  moderate-risk area. These requirements are achieved if the local distribution of the risk function is less than or equal to one. However when the habitat consists of only  high-risk areas, the disease would persist  with the infected population being uniformly distributed across the whole habitat. On the other hand, when the movement of the infected population is  restricted, Theorem \ref{theorem_stddi} indicates that the disease may persist  if the habitat has a nonempty high-risk area. Moreover, the infected population precisely occupy the high-risk sites, and the magnitude of the movement of the susceptible subgroup does not  influence the disease persistence. Our simulations in Figures \ref{fig30}-\ref{fig4} illustrated these theoretical results.

 Our above theoretical results suggest that the transmission mechanisms play an important role when the population movement of one subgroup is restricted. Indeed, when mass action mechanism is used, the persistence of the disease depends on  the average population size compared to either the average of the risk function  (when $d_S=0$) or its local distribution (when $d_I=0$). However, when standard incidence  mechanism is adopted, the disease persistence prediction depends on whether the habitat has  only   high-risk areas (when $d_S=0$) or has a nonempty high-risk area (when $d_I=0$). Interestingly, our results suggest that for either mass action  or standard incidence mechanism  disease control strategies which focus on limiting the movement of the susceptible population should be preferred { over} those focusing on restricting the movement of the infected population. 
 
 Finally,  we point out that  most of the previous works (e.g., \cite{Allen,castellano2022effect,DengWu, Peng2009,WuZou})  considered  the asymptotic profiles of the EE solutions to understand the effects of limiting population movements on the dynamics of infectious diseases, under the assumptions that the evolution of the disease happens on a fast scale compared to control strategies and  populations ultimately stabilize at the EE solutions. In this work, we study the effects of limiting population movement by focusing on the global dynamics of the degenerate systems with either $d_S=0$ or $d_I=0$,  assuming instead that the control strategies happen on much faster scale than the evolution of the disease.  %\st{We remark that the theoretical implications from these two approaches seem to align well. } 
 {  We remark that some of our results depend on the initial value. Other than that, the biological implications from these two approaches seem to align well (e.g., both approaches predict that the effectiveness of controlling the movement of susceptible people with mass action mechanism depends on the size of the total population).}

{\large\noindent{\bf Acknowledgement}}

\noindent  The authors would like to thank the reviewers for the suggestions that lead to an improvement of the paper.

{\large\noindent{\bf Declarations}}

\noindent{\bf Conflict of interest} The authors declare that they have no conflict of interest.

\bibliographystyle{plain}%{astron}%{unsrt}% {astron}% {amsplain}
\bibliography{epidem}

\end{document}